\documentclass[12pt,a4paper,leqno]{article}

\usepackage[utf8]{inputenc}
\usepackage[T1]{fontenc}
\usepackage[english]{babel}
\usepackage{amsthm}
\usepackage{amsfonts}         
\usepackage{amsmath}
\usepackage{amssymb}
\usepackage{hyperref}
\usepackage{enumerate}
\usepackage{tikz}
\usepackage{mathrsfs}  
\usetikzlibrary{matrix, arrows, decorations.pathmorphing}
\usepackage{hyperref}

\newcommand{\fref}[1]{\hyperref[{#1}]{\ref*{#1}}}

\newcommand{\B}{\mathbb{B}}
\newcommand{\Laz}{\mathbb{L}}
\newcommand{\A}{\mathbb{A}}
\newcommand{\Proj}{\mathbb{P}}

\newcommand{\Q}{\mathbb{Q}}

\newcommand{\M}{\mathbb{M}}
\newcommand{\OO}{\mathcal{O}}

\newcommand{\Z}{\mathbb{Z}}
\newcommand{\I}{\mathbb{I}}

\newcommand{\Li}{\mathscr{L}}

\newcommand{\cc}{\mathcal{C}}

\newcommand{\op}{\mathrm{op}}
\newcommand{\ch}{\mathrm{CH}}
\newcommand{\lch}{\mathrm{LCH}}

\newcommand{\coim}{\mathrm{\coim}}

\newcommand{\Hom}{\mathrm{Hom}}
\newcommand{\Spec}{\mathrm{Spec}}
\newcommand{\Gr}{\mathrm{Gr}}
\newcommand{\bl}{\mathrm{Bl}}
\newcommand{\map}{\mathrm{Map}}
\newcommand{\QCoh}{\mathrm{QCoh}}

\newcommand{\colim}{\mathrm{colim}}

\newtheorem{theo}{Tplottin ubuntuheorem}[section]
\theoremstyle{plain}
\newtheorem{thm}[theo]{Theorem}
\newtheorem{lem}[theo]{Lemma}
\newtheorem{prop}[theo]{Proposition}
\newtheorem{cor}[theo]{Corollary}
\newtheorem*{thm*}{Theorem}
\newtheorem*{lem*}{Lemma}
\newtheorem*{prop*}{Proposition}
\newtheorem*{cor*}{Corollary}

\theoremstyle{definition}
\newtheorem{defn}[theo]{Definition}
\newtheorem{cons}[theo]{Construction}

\newtheorem{ex}[theo]{Example}
\newtheorem{que}[theo]{Question}

\theoremstyle{remark}
\newtheorem{rem}[theo]{Remark}

\newcommand{\Addresses}{{
  \bigskip
  \footnotesize

  Toni Annala, \textsc{Department of Mathematics, University of British Columbia,
    Vancouver, BC V6T1Z2 Canada}\par\nopagebreak
  \textit{E-mail address:} \texttt{tannala@math.ubc.ca}

}}

\title{Bivariant derived algebraic cobordism}
\author{Toni Annala}
\date{}
\begin{document}

\maketitle
\begin{abstract}
We extend the derived Algebraic bordism of Lowrey and Schürg to a bivariant theory in the sense of Fulton and MacPherson, and establish some of its basic properties. As a special case, we obtain a completely new theory of cobordism rings of singular quasi-projective schemes. The extended cobordism is shown to specialize to algebraic $K^0$ analogously to Conner-Floyd theorem in topology. We also give a candidate for the correct definition of Chow rings of singular schemes.
\end{abstract}

\tableofcontents

\section{Introduction} 

Algebraic cobordism, first introduced by Voevodsky in the context of motivic homotopy theory, was later given an alternative description in terms of cycles and relations by Levine and Morel. The alternative description works, however, only for smooth schemes: what Levine and Morel actually construct is the corresponding homology theory $\Omega_*$, \emph{algebraic bordism}, and this gives a cohomology theory only when restricted to smooth schemes. Although initially it was not clear if the theories of Voevodsky and Levine-Morel agree,  the dichotomy was essentially resolved by Levine in \cite{Lev}, when he showed that the natural map $\Omega_*(X) \to MGL^{BM}_{2*, *}(X)$ of homology groups is an isomorphism for all $X$ quasi-projective over a field of characteristic zero. 

The study of algebraic cobordism is justified by its universality: it is supposed to play a similar role in the study of cohomology theories of algebraic varieties as complex cobordism plays in the study of of complex oriented cohomology theories in topology. As an example of this principle in action, given a quasi-projective scheme $X$ over a field of characteristic 0, we obtain, by the results of Levine-Morel, natural identifications
$$\Z_a \otimes_\Laz \Omega_*(X) \cong \ch_*(X)$$
and 
$$\Z_m \otimes_\Laz \Omega_*(X) \cong G_0(X),$$
where $\Laz \cong \Omega_*(pt)$ is the Lazard ring, and $\Z_a, \Z_m$ are the integers considered as an $\Laz$-algebra in two different ways (via the maps classifying the additive and the multiplicative formal group laws respectively). Therefore algebraic bordism completely determines both the Chow group and the Grothendieck group of coherent sheaves. Moreover, taking this point of view makes it clear that $\ch_*(X)$ and $G_0(X)$ should agree with rational coefficients, as there is essentially only one morphism $\Laz \to \Q$.

\subsection{Summary of results}

The main purpose of this paper is to extend algebraic bordism groups $\Omega_*$ constructed by Levine and Morel to a bivariant theory. Recall that a \emph{bivariant theory} $\B^*$ consists of 
\begin{enumerate}
\item an Abelian group $\B^*(X \to Y)$ for any morphism $X \to Y$;

\item bilinear products (\emph{bivariant product})
$$\bullet: \B^*(X \to Y) \times \B^*(Y \to Z) \to \B^*(X \to Z)$$
associated to compositions of morphisms;

\item \emph{pushforwards}
$$f_*: \B^*(X \to Y) \to \B^*(X' \to Y)$$
whenever $X \to Y$ factors through a \emph{confined} morphism $f: X \to X'$ (e.g. proper morphism);

\item \emph{pullbacks}
$$g^*: \B^*(X \to Y) \to \B^*(X \times_Y Y' \to Y')$$
whenever the pullback of $X \to Y$ to $Y'$ is \emph{independent} (e.g. Tor-independent);
\end{enumerate}
and these structures are furthermore required to satisfy certain compatibility axioms (see Section \fref{BivariantTheories} for more details). The main benefit of such a theory is that it encompasses both a homological theory $\Omega^*(X \to pt)$ (covariant for confined morphisms) and a ring valued cohomological theory $\Omega^*(X \to X)$ (contravariant for all morphisms), which interact naturally (by bivariant product), and satisfy many desirable properties (such as projection formula). The main goal can now be stated as finding a bivariant theory $\Omega^*$, so that the homology groups $\Omega^*(X \to pt)$ can be naturally identified with the algebraic bordism groups $\Omega_*(X)$.

The construction of $\Omega^*$ is carried in Section \fref{BivariantConstruction}, and it requires us to use derived algebraic geometry. As an end product, we obtain a bivariant theory $\Omega^*$ associating an Abelian group $\Omega^* (X \to Y)$ for any homotopy class of morphisms of quasi-projective derived schemes. In Proposition \fref{InducedHomologyIsBordism} we identify the associated homology groups $\Omega^*(X \to pt)$ with the derived bordism groups $d \Omega_*(X)$ of Lowrey and Schürg from \cite{LS2}. By the main result of \emph{loc. cit.}, there is a natural isomorphism
$$d \Omega_*(X) \cong \Omega_*(\tau_0 X),$$
where $\tau_0$ is the classical truncation of the derived scheme $X$, and we have therefore found a bivariant theory extending $\Omega_*$. 

It follows immediately from the construction that $\Omega^*$ has natural \emph{orientations} $\theta(f) \in \Omega^*(X \xrightarrow{f} Y)$ along quasi-smooth morphisms $f$. Such structure immediately yields Poincaré isomorphism maps $\Omega^*(X \to X) \to \Omega^*(X \to pt)$ whenever $X$ is smooth by the general Proposition \fref{PoincareDuality}, and therefore identifies $\Omega^*(X \to X)$ with the algebraic cobordism ring of Levine and Morel. However, we note that the cohomology rings $\Omega^*(X \to X)$ are novel whenever $X$ is not smooth.

We then go on in Section \fref{ChernClassSect} to construct Chern classes 
$$c_i(E) \in \Omega^i(X \to X)$$
for $X$ an arbitrary quasi-projective derived scheme and $E$ a vector bundle on $X$. These classes allow us to show (Theorem \fref{KTheoryCobordism}) the existence of a natural isomorphism
$$\Z_m \otimes_\Laz \Omega^*(X \to X) \cong K^0(X),$$
where $K^0(X)$ is the Grothendieck group of vector bundles on $X$, and $\Z_m$ is the integers considered as an $\Laz$-algebra via the map classifying the multiplicative formal group law $x + y - xy$. This can be understood as an extension of the previously known results of Levine-Morel, where $X$ is required to be smooth.

Finally, one can use the bivariant algebraic cobordism $\Omega^*$ in order to construct a candidate for a genuine \emph{bivariant intersection theory}. Denoting by $\Z_a$ the integers considered as an $\Laz$-algebra via the map that classifies the additive formal group law, we define
$$\ch^*(X \to Y) := \Z_a \otimes_\Laz \Omega^*(X \to Y).$$
It is a straightforward corollary of the previous results in the paper that the associated homology groups $\ch^*(X \to pt)$ are naturally identified with the usual Chow group $\ch_*(\tau_0 X)$ of the classical truncation $\tau_0 X$ of any quasi-projective derived scheme $X$. We can then use the theory of Chern classes constructed in Section \fref{ChernClassSect} in order to define a Chern character map
$$ch: K^0(X) \to \Q \otimes_\Z \ch^*(X \to X)$$
commuting with pullbacks, and whose failure to commute with pushforwards is measured by a Todd class (see Theorem \fref{DerivedGRR}). Tensoring also the source $K^0(X)$ by rational numbers, the Chern character $ch$ becomes an isomorphism of rings. Finally, if $X$ is a quasi-projective classical scheme, there is a natural map
$$\lch^*(X) \to \ch^*(X)$$
where the left hand side is the Chow ring of a singular scheme defined in \cite{Ful2} Section 3. In Proposition \fref{CHRationallyStableUnderLKE} we show that the above morphism gives an isomorphism with rational coefficients.

\subsection{Related work}

This is not the first time $\Omega_*$ is extended to a bivariant functor. In \cite{KLG}, Karu and González, inspired by the results of Anderson and Payne on operational $K$-theory in \cite{AP}, define the \emph{operational} bivariant algebraic cobordism, a theory we will from now on denote by $\op \Omega^*$. This is analogous to the bivariant operational Chow theory of Fulton and MacPherson. Although simple to define, and satisfying nice functorial properties making the calculation of operational groups feasible, the theory leaves much to hope in terms of interpretability. The elements of the group are certain collections of maps between the corresponding bordism groups and it is therefore not clear what the theory computes. In fact, it follows from the functorial properties of bivariant theories, that the operational bivariant theory bootstrapped to a homology theory merely records how certain types of bivariant extensions of the homology theory act on it. Hence, the operational bivariant theory can be thought as the universal target for bivariant extensions of a homology theory. 

Another version of bivariant algebraic cobordism has been constructed by Déglise in \cite{De}. These bivariant groups $MGL^{BM}_{2*,*}(X/Y)$ do not agree with our groups $\Omega^*(X \to Y)$ as the homotopy invariance fails in our setting (see Remark \fref{HomotopyInvarianceFails}). As of now, the existence of a Grothendieck transformation $\Omega^*(X \to Y) \to MGL^{BM}_{2*,*}(X/Y)$ seems likely, but a construction would most likely require extending $MGL^{BM}_{2*,*}(X/Y)$ to maps between derived schemes. We note that neither of the two above bivariant theories can satisfy an analogue of Theorem \fref{KTheoryCobordism}.

Besides operational $K$-theory, there exists another, more geometric, bivariant extension of the Grothendieck $G$-theory: the $K$-theory of \emph{relatively perfect complexes}. The corresponding cohomology theory is nothing more than the usual $K$-theory of vector bundles (at least for quasi-projective schemes). This theory is much more complicated than its operational counterpart, see for example \cite{Gub} where it is shown that there are toric varieties with very large $K$-groups. Another difference between the two theories is that, whereas operational $K$-theory is homotopy invariant, the $K$ theory of relatively perfect complexes is not. The Picard groups $\mathrm{Pic}(X)$ exhibits similar behavior: there are well known examples of singular schemes $X$ such that $\mathrm{Pic}(X)$ is not isomorphic to $\mathrm{Pic}(\A^1 \times X)$. As one would expect that the Picard group is the degree one part of some conjectural Chow cohomology ring $\ch^*$, one is tempted to draw the conclusion that failure of homotopy invariance is to be expected from cohomology rings of singular varieties.

\subsection{On the characteristic 0 assumption}

Many results of this article need the characteristic 0 assumption. Like the construction of $d \Omega_*$ in \cite{LS2}, the construction of our bivariant algebraic cobordism $\Omega^*$ makes perfect sense without any restrictions on the characteristic. However, without the characteristic 0 assumption, the homology theory $d \Omega^*$ does not satisfy desirable properties, and as the main strategy of this paper is to show that desirable properties of $\Omega^*$ follow from those of $d \Omega_*$, we don't obtain many results in positive characteristic.

For example, outside characteristic 0, we can not construct a theory of Chern classes in the sense of Section \fref{ChernClassSect}. But the theory is needed in order to prove the Theorem \fref{KTheoryCobordism} ($K$-theory is a specialization of $\Omega^*$), and in order to prove the extension of Grothendieck-Riemann-Roch as in Theorem \fref{DerivedGRR}. 

However, it seems likely that a similar strategy would work after replacing $d \Omega_*$ by some other model for derived algebraic bordism (perhaps coming from motivic homotopy theory), and then imitating the proofs and constructions from Sections \fref{MainConstructionSect} to \fref{BivariantChowSect}.

\subsection{Acknowledgements}

The author would like to thank his advisor Kalle Karu for suggesting the topic and for helpful discussions. The author would also like to thank Samuel Bach and Kai Behrend for answering discussions about derived algebraic geometry, and Shoji Yokura for multiple comments. Both the anonymous referees gave comments that have significantly improved the exposition. Special thanks are due to the second referee who brought into the authors attention the work done in \cite{Ful2}, and without whom there would be neither Proposition \fref{CHRationallyStableUnderLKE} nor Question \fref{IsCHStableUnderLKE}.

\section{Background}
In this section, we review the technical background material necessary for the rest of the paper: namely, the general framework of bivariant theories, as well as some results from derived algebraic geometry. 

\subsection{Bivariant theories}\label{BivariantTheories}

The notion of a \emph{bivariant theory} was introduced by Fulton and MacPherson in \cite{FM} to unify multiple Grothendieck-Riemann-Roch type theorems. As a simultaneous generalization of both homology and cohomology theories, the bivariant theory assigns an Abelian group to any morphism between the spaces of interest.

More rigorously, suppose we have a category $\mathcal{C}$, which is assumed to have a final object $pt$ and all fibre products, together with a class of \emph{confined maps} which is closed under composition and pullback, and contains all isomorphisms (e.g. proper morphisms) The category should also come equipped with a class of Cartesian squares called \emph{independent squares} that are ''closed under composition'' and ''contain all identities'', the two conditions meaning that whenever the two smaller squares in
\begin{center}
\begin{tikzpicture}[scale=2]
\node (A1) at (0,1) {$X'$};
\node (B1) at (1,1) {$Y'$};
\node (A2) at (0,0) {$X$};
\node (B2) at (1,0) {$Y$};
\node (C2) at (2,0) {$Z$};
\node (C1) at (2,1) {$Z'$};
\node (OR) at (2.6,0.5) {or};
\node (A1') at (4.2,-0.5) {$X$};
\node (B1') at (4.2,0.5) {$Y$};
\node (A2') at (3.2,-0.5) {$X'$};
\node (B2') at (3.2,0.5) {$Y'$};
\node (C2') at (3.2,1.5) {$Z'$};
\node (C1') at (4.2,1.5) {$Z$};
\path[every node/.style={font=\sffamily\small}]
(A1) edge[->] (A2)
(B1) edge[->] (B2)
(A1) edge[->] (B1)
(A2) edge[->] (B2)
(B2) edge[->] (C2)
(B1) edge[->] (C1)
(C1) edge[->] (C2)
(A1') edge[<-] (A2')
(B1') edge[<-] (B2')
(A1') edge[<-] (B1')
(A2') edge[<-] (B2')
(B2') edge[<-] (C2')
(B1') edge[<-] (C1')
(C1') edge[<-] (C2')
;
\end{tikzpicture}
\end{center}
are independent, then so is the outer square, and that all the squares of form
\begin{center}
\begin{tikzpicture}[scale=2]
\node (A1) at (0,1) {$X$};
\node (B1) at (1,1) {$Y$};
\node (A2) at (0,0) {$X$};
\node (B2) at (1,0) {$Y$};
\node (A1') at (2.2,1) {$X$};
\node (B1') at (3.2,1) {$X$};
\node (A2') at (2.2,0) {$Y$};
\node (B2') at (3.2,0) {$Y$};
\node (AND) at (1.6,0.5) {and};
\path[every node/.style={font=\sffamily\small}]
(A1) edge[->] node[right]{$1$} (A2)
(B1) edge[->] node[right]{$1$} (B2)
(A1) edge[->] node[above]{$f$} (B1)
(A2) edge[->] node[above]{$f$} (B2)
(A1') edge[->] node[right]{$f$} (A2')
(B1') edge[->] node[right]{$f$} (B2')
(A1') edge[->] node[above]{$1$} (B1')
(A2') edge[->] node[above]{$1$} (B2')
;
\end{tikzpicture}
\end{center}
are independent. Note that the definition does not make any assumptions about the symmetry of such squares. The ''transpose'' of an independent square is not assumed to be independent. However, in the practical cases of interest in algebraic geometry, and certainly in all cases appearing in this paper, the independent squares are indeed symmetric in this sense, and hence the reader need not worry about such subtleties.

The following examples are perhaps the most relevant that occur in the classical literature:

\begin{ex}
Suppose $\mathcal{C}$ is the category of quasi-projective schemes over a ground field $k$. We can now choose proper morphisms to be the confined morphisms, and either all fibre squares, or only the Tor-independent ones, to be the independent squares. It is easy to check that both possible choices of independent squares satisfy the above conditions.
\end{ex}

A bivariant theory $\B^*$ on such a category (the extra structure is implicit from now on) assigns a (graded) Abelian group $\B^*(X \to Y)$ to \emph{any morphism} $X \to Y$, no confinedness required. The purpose of the special classes of maps and squares is to give rise to operations of the bivariant theory. Namely, whenever a morphism $X \to Y$ factors as $X \stackrel{f}{\to} X' \to Y$, and $f$ is confined, we have an induced \emph{pushforward morphism}
\begin{equation*}
f_*: \B^*(X \to Y) \to \B^*(X' \to Y),
\end{equation*}
and whenever we have an independent square 
\begin{center}
\begin{tikzpicture}[scale=2]
\node (A1) at (0,1) {$X'$};
\node (B1) at (1,1) {$X$};
\node (A2) at (0,0) {$Y'$};
\node (B2) at (1,0) {$Y$};
\path[every node/.style={font=\sffamily\small}]
(A1) edge[->] (A2)
(B1) edge[->] node[right]{$f$} (B2)
(A1) edge[->] (B1)
(A2) edge[->] node[above]{$g$} (B2)
;
\end{tikzpicture}
\end{center}
we have an induced \emph{pullback morphism} 
\begin{equation*}
g^*:  \B^*(X \to Y) \to \B^*(X' \to Y')
\end{equation*}
These operations are assumed to be functorial in the obvious sense.

The third bivariant operation is the \emph{bivariant product}. Whenever we have composable maps $X \to Y \to Z$, we have a bilinear map
\begin{equation*}
\bullet: \B^*(X \to Y) \times \B^*(Y \to Z) \to \B^*(X \to Z). 
\end{equation*}
This product is assumed to be associative and unital, the latter requirement meaning that in $\B^*(X \stackrel{1}{\to} X)$ there is an element $1_X$ which we assume to act by identities in all left and right multiplications.

These operations are required to satisfy four compatibility properties. In all diagrams encircled symbols near arrows will denote elements of bivariant groups, and symbols without circles will, as usual, simply give a name for the corresponding maps.
\begin{itemize}
\item[($A_{12}$)] \emph{Pushforward and the bivariant product commute.} In the following situation 
\begin{center}
\begin{tikzpicture}[scale=2]
\node (A1) at (-1,1) {$X$};
\node (B1) at (0,1) {$Y$};
\node (C1) at (1,1) {$Z$};
\node (D1) at (2,1) {$W$};
\path[every node/.style={font=\sffamily\small}]
(A1) edge[->] node[above]{$f$} (B1)
(B1) edge[->] (C1)
(C1) edge[->] node[yshift=0.4cm,circle,draw,inner sep=0pt, minimum size=0.5cm]{$\beta$} (D1)
(A1) edge[->,bend right] node[yshift=-0.4cm,circle,draw,inner sep=0pt, minimum size=0.5cm]{$\alpha$} (C1)
;
\end{tikzpicture}
\end{center}
we have that $f_*(\alpha \bullet \beta) = f_*(\alpha)\bullet\beta$.

\item[($A_{13}$)] \emph{Pullback and the bivariant product commute.} In the following situation
\begin{center}
\begin{tikzpicture}[scale=2]
\node (A2) at (1,2) {$X$};
\node (B2) at (1,1) {$Y$};
\node (C2) at (1,0) {$Z$};
\node (A1) at (0,2) {$X'$};
\node (B1) at (0,1) {$Y'$};
\node (C1) at (0,0) {$Z'$};
\path[every node/.style={font=\sffamily\small}]
(A1) edge[->] (A2)
(B1) edge[->] node[above]{$h'$} (B2)
(C1) edge[->] node[above]{$h$} (C2)
(A1) edge[->] (B1)
(B1) edge[->] (C1)
(A2) edge[->] node[xshift=0.4cm,circle,draw,inner sep=0pt, minimum size=0.5cm]{$\alpha$} (B2)
(B2) edge[->] node[xshift=0.4cm,circle,draw,inner sep=0pt, minimum size=0.5cm]{$\beta$} (C2)
;
\end{tikzpicture}
\end{center}
whenever the two small squares (and hence the large square) are independent, we have that $h^*(\alpha \bullet \beta) = h'^*(\alpha) \bullet h^*(\beta)$.

\item[($A_{23}$)] \emph{Pullback and pushforward commute.} In the following situation
\begin{center}
\begin{tikzpicture}[scale=2]
\node (A2) at (1,2) {$X$};
\node (B2) at (1,1) {$Y$};
\node (C2) at (1,0) {$Z$};
\node (A1) at (0,2) {$X'$};
\node (B1) at (0,1) {$Y'$};
\node (C1) at (0,0) {$Z'$};
\path[every node/.style={font=\sffamily\small}]
(A1) edge[->] (A2)
(B1) edge[->] (B2)
(C1) edge[->] node[above]{$h$} (C2)
(A1) edge[->] node[right]{$f'$} (B1)
(B1) edge[->] (C1)
(A2) edge[->] node[right]{$f$} (B2)
(B2) edge[->] (C2)
(A2) edge[->,bend left] node[xshift=0.4cm,circle,draw,inner sep=0pt, minimum size=0.5cm]{$\alpha$} (C2)
;
\end{tikzpicture}
\end{center}
whenever $f$ is confined and the big square and the lower small square are independent, we have that $h^*f_*(\alpha) = f'_*h^*(\alpha)$.

\item[($A_{123}$)] \emph{Push-pull formula.} In the following situation
\begin{center}
\begin{tikzpicture}[scale=2]
\node (A2) at (1,2) {$X$};
\node (B2) at (1,1) {$Y$};
\node (A1) at (0,2) {$X'$};
\node (B1) at (0,1) {$Y'$};
\node (B3) at (2,1) {$Z$};
\path[every node/.style={font=\sffamily\small}]
(A1) edge[->] node[above]{$g'$} (A2)
(B1) edge[->] node[above]{$g$} (B2)
(A1) edge[->] (B1)
(A2) edge[->] node[xshift=0.4cm,,circle,draw,inner sep=0pt, minimum size=0.5cm]{$\alpha$} (B2)
(B2) edge[->] (B3)
(B1) edge[->,bend right] node[yshift=-0.4cm,circle,draw,inner sep=0pt, minimum size=0.5cm]{$\beta$} (B3)
;
\end{tikzpicture}
\end{center}
whenever the square is independent and the morphism $g$ is confined, we have that $g'_*(g^*(\alpha) \bullet \beta) = \alpha \bullet g_*(\beta)$.
\end{itemize}

From the requirements for the classes of confined morphisms and independent squares it follows that the groups $\B^*(X \to pt)$ form a covariant functor for \emph{confined morphisms} induced by the bivariant pushforward, and $\B^*(X \stackrel{1}{\to} X)$ form a contravariant functor with respect to \emph{all morphism} induced by the bivariant pullback associated to independent squares. The former of these is the \emph{homology theory} induced by the bivariant theory, and similarly the latter one is the \emph{cohomology theory}. These are denoted by $\B_*(X)$ and $\B^*(X)$ respectively, with grading conventions depending on the context. One immediately observes that bivariant product makes the cohomology groups $\B^*(X)$ rings, and the homology groups $\B_*(X)$ modules over the cohomology. This was one of the motivations behind the notion of bivariant groups, and corresponds to the fact that many theories come as a pair of cohomology and homology theory such that the cohomology is a ring (cup product) acting on the homology (cap product). 

To arrive at the notion of an \emph{orientation}, we need another special class of maps of $\cc$ this time called \emph{specialized morphisms} (\cite{Yo1}, Definition 2.2 (i)), which is assumed to contain all isomorphisms and to be closed under composition (but not necessarily under pullback). An \emph{orientation} with respect to this class of morphism is an assignment $\theta (f) \in \B^*(X \stackrel f \to Y)$ for all specialized morphisms $f$. We will also assume that this orientation is multiplicative, i.e., if we have composable specialized maps $f$ and $g$, then $\theta(f) \bullet \theta(g) = \theta(g \circ f)$, and $\theta(\mathrm{id}) = 1$. The following extra condition is satisfied by some, but not all bivariant theories.

\begin{defn}[Cf. \cite{Yo1} Definition 2.2 (ii)]\label{NiceOrientation}
Suppose that the specialized morphisms are stable under pullbacks in independent squares. We say that a orientation $\theta$ is \emph{stable under pullbacks} or \emph{nice}, if for every independent square
\begin{center}
\begin{tikzpicture}[scale=2]
\node (A1) at (0,1) {$X'$};
\node (B1) at (1,1) {$X$};
\node (A2) at (0,0) {$Y'$};
\node (B2) at (1,0) {$Y$};
\path[every node/.style={font=\sffamily\small}]
(A1) edge[->] node[right]{$f'$} (A2)
(B1) edge[->] node[right]{$f$} (B2)
(A1) edge[->] (B1)
(A2) edge[->] node[above]{$g$} (B2)
;
\end{tikzpicture}
\end{center}
with $f$ (and therefore also $f'$) specialized, we have $g^*(\theta(f)) = \theta(f')$.
\end{defn}

Let us now look at some classical examples.

\begin{ex}[Bivariant intersection theory, \cite{Ful} Chapter 17]
Let $\mathcal{C}$ be again the category of quasi-projective $k$-varieties together with proper morphisms as confined morphisms and all fibre squares as independent squares. The \emph{operational bivariant Chow theory} is the bivariant theory that associates to a morphism $f: X\to Y$ the graded Abelian group, so that a degree $i$ element $c$ of $\op CH^*(X \to Y)$ (a degree $i$ bivariant class) is
\begin{enumerate}
\item a \emph{collection of homomorphisms} $c_g: \ch_*(Y') \to \ch_{*-i}(X')$ for all morphisms $g: Y' \to Y$, where $X'$ is the pullback of $X$ along $g$;

\item such that the morphisms $c_g$ commute with pullbacks along flat morphisms, pushforwards along proper morphisms and refined l.c.i. pullbacks (see Definition 17.1 of \emph{loc. cit.} for a more details).
\end{enumerate}
The bivariant operations are defined as follows:
\begin{enumerate}

\item If $X \to Z$ factors through a proper morphism $f: X \to Y$, and $c \in \op \ch^*(X \to Z)$ is a bivariant class, then the bivariant ppushforward $f_*(c) \in \op\ch^*(Y \to Z)$ is given by the collection
$$f_*(c)_g := f'_* c_g,$$
where $f': X' \to Z'$ is the pullback of $f$ in the Cartesian diagram induced by $g: Y' \to Y$.

\item Given a Cartesian square
\begin{center}
\begin{tikzpicture}[scale=2]
\node (A1) at (0,1) {$X'$};
\node (B1) at (1,1) {$Y'$};
\node (A2) at (0,0) {$X$};
\node (B2) at (1,0) {$Y$};
\path[every node/.style={font=\sffamily\small}]
(A1) edge[->] (A2)
(B1) edge[->] node[right]{$g$} (B2)
(A1) edge[->] (B1)
(A2) edge[->] (B2)
;
\end{tikzpicture}
\end{center}
and a bivariant class $c \in \op\ch^*(X \to Y)$, then the bivariant pullback $g^*(c) \in \op\ch^*(X \to Y)$ is given by the collection
$$g^*(c)_h := c_{g \circ h},$$
where $h$ is a morphism $Y'' \to Y$.

\item The bivariant product is induced by composition of homomorphisms in the obvious way.
\end{enumerate}
In this example the induced homology groups $\op\ch_*(X)$ are isomorphic with the usual Chow groups $\ch_*(X)$. The induced cohomology rings $\op\ch^*(X)$ are usually called \emph{operational Chow rings}.

Moreover, by letting the class of specialized morphisms to be l.c.i. morphisms of pure relative dimension, we can orient the above theory. Indeed, let $f: X \to Y$ be an l.c.i. morphism of relative dimension $d$. We can now define the orientation $\theta^{\ch}(f) \in \ch^{-d}(X \to Y)$ by 
$$\theta^{\ch}(f)_g := f_g^!,$$
where the right hand side is the \emph{refined Gysin pullback morphism} associated to the morphism $g: Y' \to Y$, defined as in \cite{Ful} Section 6.6. 

As l.c.i. morphisms are not stable under pullbacks, the orientation $\theta^{\ch}$ is not nice. However, l.c.i. morphisms are stable under Tor-independent pullbacks, and moreover, restricting $\theta^{\ch}$ to the subtheory whose independent squares are the Tor-independent fibre squares makes it stable under pullbacks.
\end{ex}

The next example is a more genuine of a bivariant theory in the sense that it's definition is completely geometric.

\begin{ex}[Bivariant $K^0$, \cite{FM} Part II Chapter 1]
Let $\mathcal{C}$ be the category of quasi-projective schemes over $k$ with proper morphisms as confined morphisms and Tor-independent Cartesian squares as independent squares. Given a morphism $f: X \to Y$, a complex $\mathcal{F}_\bullet$ of quasi-coherent sheaves on $X$ is called \emph{$f$-perfect} if for any (and therefore all) factorization of $f$ as $p \circ i$ with $p: Z \to Y$ smooth and $i: X \to Z$ a closed embedding, the pushforward $i_*(\mathcal{F}_\bullet)$ is quasi-isomorphic to a finite complex of vector bundles on $Z$ (i.e., $i_*(\mathcal{F}_\bullet)$ is a \emph{perfect complex}). We can define $K^0(X \to Y)$ to be the Abelian group generated by quasi-isomorphism classes of $f$-perfect complexes modulo the relation
$$[\mathcal{F}_\bullet] = [\mathcal{F}'_\bullet] + [\mathcal{F}''_\bullet]$$
whenever we have a cofibre sequence
$$\mathcal{F}'_\bullet \to \mathcal{F}_\bullet \to \mathcal{F}''_\bullet.$$

These groups give rise to a bivariant theory where pushforwards along proper morphisms and pullbacks along Tor-independent squares are given by the derived pushforward and derived pullback of sheaves respectively. Moreover, if $[\mathcal{F}_\bullet] \in K^0(X \stackrel{f}{\to} Y)$ and $[\mathcal{G}_\bullet] \in K^0(Y \to Z)$, then we can define
$$[\mathcal{F}_\bullet] \bullet [\mathcal{G}_\bullet] := [\mathcal{F}_\bullet \otimes^L_{\OO_X} Lf^* \mathcal{G}_\bullet] \in K^0(X \to Z).$$
The induced cohomology rings and homology groups of the bivariant $K^0$ are naturally identified as the Grothendieck ring $K^0(X)$ of vector bundles, and the Grothendieck group $G_0(X)$ of coherent sheaves respectively.

Moreover, we can orient the bivariant algebraic $K$-theory along morphisms of finite Tor-dimension (e.g. l.c.i. morphisms). It is known that $f: X \to Y$ has finite Tor-dimension if and only if the structure sheaf $\OO_X$ is an $f$-perfect complex. Therefore, we can define an orientation $\theta^{K}$ by the formula
$$\theta^{K}(f) := [\OO_X] \in K^0(X \to Y).$$
The orientation $\theta^K$ is clearly stable under pullback. 
\end{ex}


Any element $\alpha \in \B^*(X \stackrel f \to Y)$ gives rise to a \emph{Gysin pullback morphism} $\alpha^!: \B_*(Y) \to \B_*(X)$ by sending $\beta$ to $\alpha \bullet \beta$, and a \emph{Gysin pushforward morphism} $\alpha_!: \B^*(X) \to \B^*(Y)$ by sending $\gamma$ to $f_*(\gamma \bullet \alpha)$ whenever $f$ is confined. In the special cases when $f$ is also specialized, we will use the shorthand notation $f^!$ and $f_!$ for the Gysin morphisms $\theta(f)^!$ and $\theta(f)_!$ respectively, and it is an exercise in bivariant formalism to show that these operations will yield functorial wrong way maps $f^!$ for homology and $f_!$ cohomology.

\begin{ex}
In $\op\ch^*$, the Gysin pullback morphisms $f^!: \ch_*(Y) \to \ch_*(X)$ induced by the orientation of the l.c.i. morphism $f: X \to Y$ coincides with the usual Gysin pullback constructed in Chapter 6 of \cite{Ful}. 

For bivariant $K^0$, given a morphism $f: X \to Y$ of finite Tor-dimension (which, we recall, should be proper for Gysin pushforward), the induced Gysin morphisms $f_!: K^0(X) \to K^0(Y)$ and $f^!: G_0(Y) \to G_0(X)$ are given by formulas
$$[\mathcal{F}_\bullet] \mapsto [Rf_*\mathcal{F}_\bullet]$$
and
$$[\mathcal{G}_\bullet] \mapsto [Lf^*\mathcal{G}_\bullet]$$
respectively. 
\end{ex}

\subsubsection*{Some properties of bivariant theories}

The axioms of a bivariant theory allow great generality, and hence it is beneficial to discuss some properties that the theories we will be interested about exhibit. The first of these is commutativity.

\begin{defn}\label{CommutativeTheory}
A bivariant theory $\B^*$ is \emph{commutative}, if, whenever we are in the situation
\begin{center}
\begin{tikzpicture}[scale=2]
\node (A2) at (1,2) {$X$};
\node (B2) at (1,1) {$Y$};
\node (A1) at (0,2) {$X'$};
\node (B1) at (0,1) {$Y'$};
\path[every node/.style={font=\sffamily\small}]
(A1) edge[->] (A2)
(B1) edge[->] node[above]{$g$} node[yshift=-0.4cm,circle,draw,inner sep=0pt, minimum size=0.5cm]{$\beta$} (B2)
(A1) edge[->] (B1)
(A2) edge[->] node[left]{$f$} node[xshift=0.4cm,circle,draw,inner sep=0pt, minimum size=0.5cm]{$\alpha$} (B2)
;
\end{tikzpicture}
\end{center}
and the square (and its transpose) are independent, the products $g^*(\alpha) \bullet \beta$ and $f^*(\beta) \bullet \alpha$ coincide in $\B^*(X' \to Y)$.
\end{defn}

\begin{rem}
The bivariant theories $\op\ch^*$ and $K^0$ we introduced earlier are bivariant, as will be all the other bivariant theories we are going to consider in this article. 
\end{rem}

\begin{defn}\label{StrongDefn}
Let $\B^*$ be a bivariant theory with an orientation $\theta$. We say that the orientation $\theta(f)$ of $f : Y \to Z$ is \emph{strong} if the maps 
\begin{equation*}
- \bullet \theta(f): \B^*(X \to Y) \to \B^*(X \to Z)
\end{equation*}
are isomorphisms for all $X \to Y$.
\end{defn}

The following proposition shows some non-trivial interplay between commutativity, having nice orientations in the sense of Definition \fref{NiceOrientation}, and having strong orientations along some subclass of morphisms.

\begin{prop}\label{StrongOrientations}
Let $\B^*$ be a commutative bivariant theory with nice orientation. Suppose there is a subclass of specialized morphisms closed under compositions and pullbacks called \emph{specialized projections} satisfying that whenever $g \circ f = 1$ and $g$ is a specialized projection, $f$ is specialized. Moreover, suppose that all Cartesian squares pulling back a specialized projection are independent, and that the independent squares satisfy the \emph{right cancellation property}. Now the orientation of $\B^*$ is strong along specialized projections.
\end{prop}
\begin{rem}
The right cancellation property simply requires that when we have either
\begin{center}
\begin{tikzpicture}[scale=2]
\node (A1) at (0,1) {$X'$};
\node (B1) at (1,1) {$Y'$};
\node (A2) at (0,0) {$X$};
\node (B2) at (1,0) {$Y$};
\node (C2) at (2,0) {$Z$};
\node (C1) at (2,1) {$Z'$};
\node (OR) at (2.6,0.5) {or};
\node (A1') at (4.2,-0.5) {$X$};
\node (B1') at (4.2,0.5) {$Y$};
\node (A2') at (3.2,-0.5) {$X'$};
\node (B2') at (3.2,0.5) {$Y'$};
\node (C2') at (3.2,1.5) {$Z'$};
\node (C1') at (4.2,1.5) {$Z$};
\path[every node/.style={font=\sffamily\small}]
(A1) edge[->] (A2)
(B1) edge[->] (B2)
(A1) edge[->] (B1)
(A2) edge[->] (B2)
(B2) edge[->] (C2)
(B1) edge[->] (C1)
(C1) edge[->] (C2)
(A1') edge[<-] (A2')
(B1') edge[<-] (B2')
(A1') edge[<-] (B1')
(A2') edge[<-] (B2')
(B2') edge[<-] (C2')
(B1') edge[<-] (C1')
(C1') edge[<-] (C2')
;
\end{tikzpicture}
\end{center}
and when the large square and the right (lower) square are independent, then so is the remaining little square. 
\end{rem}
\begin{proof}
The proof is just a formalization of the proof of Proposition 17.4.2. in \cite{Ful}.
\end{proof}

\begin{rem}
The Proposition 17.4.2. in \cite{Ful} is a special case of the above proposition. Indeed, one may regard the operational Chow-theory as a commutative bivariant theory having orientations along l.c.i. morphisms, and pullbacks along all Tor-independent squares. The orientations are stable under pullbacks in such squares. It also is easily verified that Tor-independence satisfies the right cancellation property as defined above. Smooth maps play the part of specialized projections: indeed, the pullback of a smooth morphism is smooth, and whenever $X \stackrel g \to X' \stackrel f \to X = 1$ and $f$ is smooth, then the (necessarily) closed immersion $g$ is locally a complete intersection, and therefore an l.c.i. morphism. Hence all the assumptions of the above proposition hold. It does not matter that one may regard $\op \ch^*$ as a bivariant theory having pullbacks along all Cartesian squares, and that the orientations are no longer necessarily stable under pullbacks in such squares.
\end{rem}

There is also an immediate generalization of Poincaré duality for objects $X$ where the map $X \to pt$ to the final object of the category is a specialized projection. The orientation of $X \to pt$ induces an isomorphism $\B^*(X \to X) \to \B^*(X \to pt)$, as is shown by the above proposition. What is not immediately clear however, is that in this situation we have two ways to define the a product on the isomorphic Abelian groups, and that the products coincide. Indeed, in this situation both of the projections $X \times X \to X$ are specialized projections, and hence the diagonal embedding $\Delta: X \to X \times X$ is specialized. We can define the \emph{intersection product} $\frown$ on $\B_*(X)$ by setting $\alpha_1 \frown \alpha_2 = \Delta^!(\alpha_1 \times \alpha_2)$. In this case we have 

\begin{prop}[Poincaré duality]\label{PoincareDuality}
Let $\B^*$ be a commutative bivariant theory as in \fref{StrongOrientations}. Suppose $\pi: X \to pt$ is a specialized projection. Now the map
\begin{equation*}
\bullet \theta(\pi) : (\B^*(X), \bullet) \to (\B_*(X), \frown)
\end{equation*}
is an isomorphism of rings.
\end{prop}
\begin{proof}
The proof is merely a formalization of Corollary 17.4 in \cite{Ful}.
\end{proof}
 
\subsubsection*{Universal bivariant theory}
The main contribution of Yokura's paper \cite{Yo1} was to give a simple construction for the universal bivariant theory of a certain type in a large number of interesting cases. Namely, we must assume that the specialized morphisms are stable under pullbacks, and restrict our attention to theories that are equipped with a nice orientation in the sense of Definition \fref{NiceOrientation}. Moreover, we require all Cartesian squares obtained by pulling back confined morphisms to be independent.

Under these conditions, we have a simple construction for $\M^*$ -- the universal bivariant theory equipped with nice orientation -- which we are going to outline shortly. The universality means that for any other such a bivariant theory $\B^*$, there is a unique \emph{Grothendieck transformation} $\omega : \M^* \to \B^*$, i.e., a collection of maps 
\begin{equation*}
\omega_{X \to Y}: \M^*(X \to Y) \to \B^*(X \to Y)
\end{equation*}
compatible with all three bivariant operations in an obvious sense, and which respect orientations.

As promised, the construction of the bivariant theory $\M^*$ is very simple. Namely, the group $\M^* (X \to Y)$ is the free Abelian group generated by the isomorphism classes of confined maps $V \to X$ such that $V \to X \to Y$ is specialized. From now on, the class of such a cycle is denoted by $[V \to X]$.

The construction of the bivariant operations is fairly simple as well. If $X \to Y$ factorizes as $X \stackrel f \to X' \to Y$, where $f$ is confined, then the pushforward is defined by linearly extending
\begin{equation*}
f_*[V \to X] = [V \to X'].
\end{equation*}
If we have a map $Y' \stackrel g \to Y$ such that the induced Cartesian square is independent, then we may define the pullback $g^*$ as the linear extension of map
\begin{equation*}
g^*[V \to X] = [V' \to X']
\end{equation*}
where $V'$ and $X'$ are the pullbacks of $V$ and $X$ respectively. Note how the property of specialized morphisms being stable under pullback becomes important for the definition of pullback.

The third bivariant operation, the bivariant product, is the most complicated to define. Suppose we have elements
\begin{equation*}
\alpha = [V \to X] \in \M^*(X \to Y)
\end{equation*}
and
\begin{equation*}
\beta = [W \to Y] \in \M^*(Y \to Z).
\end{equation*}
In order to define their bivariant product $\alpha \bullet \beta \in \M^*(X \to Z)$, we begin with the following pullback diagram
\begin{center}
\begin{tikzpicture}[scale=2]
\node (A1) at (0,1) {$V'$};
\node (B1) at (1,1) {$X'$};
\node (A2) at (0,0) {$V$};
\node (B2) at (1,0) {$X$};
\node (C1) at (2,1) {$W$};
\node (C2) at (2,0) {$Y$};
\node (D2) at (3,0) {$Z$};

\path[every node/.style={font=\sffamily\small}]
(A1) edge[->] (A2)
(B1) edge[->] (B2)
(A1) edge[->] (B1)
(A2) edge[->] (B2)
(B2) edge[->] (C2)
(B1) edge[->] (C1)
(C1) edge[->] (C2)
(C2) edge[->] (D2)
;
\end{tikzpicture}
\end{center}
The map $V' \to W$ is specialized as a pullback of a specialized morphism, and therefore $V' \to Z$ is specialized as well. On the other hand both $V' \to X'$ and $X' \to X$ are confined as pullbacks of confined morphisms, and hence $V' \to X$ is confined. Hence $[V' \to X]$ is an element of the bivariant group $\M^*(X \to Z)$, and extending bilinearly yields us the bivariant product.

We will not prove that these operations satisfy the properties necessary for $\M^*$ to be called a bivariant theory; the interested reader is encouraged to consult \cite{Yo1}. Given another such a bivariant theory $\B^*$ the unique map $\M^* \to \B^*$ is easy to describe. Namely, given a bivariant cycle $[V \stackrel \alpha \to X] \in \M^*(X \stackrel f \to Y)$, it is easy to see that it must be sent to 
\begin{equation*}
\alpha_* (f \circ \alpha)^* 1_Y \in \B^*(X \to Y).
\end{equation*}
This map is, in fact, well defined, as is shown by Yokura. We also note that the theory is commutative. The proofs are omitted.

\subsection{Derived algebraic geometry and quasi-smooth morphisms}

The purpose of this subsection is to recall results from derived algebraic geometry necessary for the rest of the paper. We are not going to be very detailed; the interested reader is encouraged to consult the vast literature devoted to the subject for further details. First, however, we have to remind ourselves what derived algebraic geometry is. It is a homotopical enhancement of usual algebraic geometry in the sense that a \emph{derived scheme} is locally modeled by by spectra associated to simplicial commutative rings, which are considered up to homotopy equivalence (but remembering homotopies). If working over a field $k$ of characteristic 0, then the homotopy theory of simplicial commutative $k$-algebras agrees with those of connective commutative $k$ differential graded algebras, and connective $\mathbb{E}_\infty$ ring spectra over $k$, and therefore one may use either of these two to give local models for derived schemes without any change to the theory. 

One of the novelties of derived algebraic geometry is that the derived schemes form an \emph{$\infty$-category}. Therefore, given derived schemes $X$ and $Y$, the collection of morphisms $\map(X, Y)$ is not just a set but a space (again, considered suitably up to homotopy equivalence), which we are going to call the \emph{space of maps from $X$ to $Y$}. These mapping space often arise from a simplicial model category that gives rise to the $\infty$-category. Any scheme is also a derived scheme, and given two classical schemes $X$ and $Y$, then the mapping space $\map(X,Y)$ is naturally equivalent to the discrete space (i.e., to the set) $\map_{\mathrm{classical}}(X,Y)$ of maps of schemes in the classical sense. We say that the theory of schemes sits fully faithfully inside the theory of derived schemes. There is also a way going from derived schemes to schemes, and it is based on the fact that given a simplicial commutative ring $A$, the set of its path components $\pi_0(A)$ naturally inherits the structure of a commutative ring. Given a derived scheme $X$, the spectra of the path component rings glue together to give rise to a scheme $\tau_0(X)$ --- the \emph{(classical) truncation} of $X$. We note that $\tau_0$ is a functor from derived schemes to classical schemes.

The derived analogue of the usual fibre product of schemes is the \emph{homotopy fibre product} $X \times_Z^h Y$ of derived schemes. On the underlying classical schemes this is just the usual fibre product of schemes
\begin{equation*}
\tau_0 (X \times_Z^h Y) = \tau_0(X) \times_{\tau_0(Z)} \tau_0(Y)
\end{equation*}
but we must note that the homotopy fibre product of classical schemes needs not to coincide with the fibre product (this requires the diagram to be Tor-independent). As one would expect, the homotopy fibre product is affine locally modeled by the \emph{derived tensor product} of simplicial rings.

Any derived scheme $X$ has a naturally associated $\infty$-category $\QCoh(X)$ of \emph{quasi-coherent sheaves on $X$}. On an affine open $\Spec(A) \subset X$, this is modeled by spectrum objects in the category of modules over the simplicial ring $A$. If $X$ is a classical scheme, then $\QCoh(X)$ is naturally identified with an $\infty$-categorical enhancement of the unbounded derived category of quasi-coherent sheaves. Moreover, in this case being a spectrum object coincides with having non vanishing negative homology sheaves (an the idea is the same in the derived case).

For any map $X \to Y$ of derived schemes we have the \emph{cotangent complex} $L_{X/Y} \in \QCoh(X)$. The cotangent complex is stable under pullbacks: if
\begin{center}
\begin{tikzpicture}[scale=2]
\node (A2) at (1,2) {$Y'$};
\node (B2) at (1,1) {$Y$};
\node (A1) at (0,2) {$X'$};
\node (B1) at (0,1) {$X$};
\path[every node/.style={font=\sffamily\small}]
(A1) edge[->] (A2)
(B1) edge[->]  (B2)
(A1) edge[->] (B1)
(A2) edge[->] (B2)
;
\end{tikzpicture}
\end{center}
is homotopy Cartesian, then the derived pullback of $L_{X/Y}$ along the map $X' \to X$ coincides with $L_{X'/Y'}$. A notion of crucial importance for the paper is that of a quasi-smooth morphism. A morphism $f: X \to Y$ of derived schemes is \emph{quasi-smooth} if the truncation $\tau_0 f$ is locally of finite presentation (in the classical sense), and if the cotangent complex $L_{X/Y}$ is perfect and has Tor-amplitude in $[1,0]$, i.e., locally $L_{X/Y}$ is equivalent to the mapping cone of two free and finitely generated free quasi-coherent sheaves. It is a fact that a map between classical schemes is quasi-smooth if and only if it is l.c.i. The main benefit of the derived notion is that the homotopy pullback of a quasi-smooth morphism is quasi-smooth. As one would expect from the classical case, if we have maps $X \to Y \to Z$ of derived schemes, and $Y \to Z$ is smooth, then $X \to Z$ is quasi-smooth if and only if $X \to Y$ is quasi-smooth.

As a final note, we need to discuss some finer details of the theory of derived schemes. First of all, the concept of a \emph{vector bundle} on $X$ makes perfect sense: it is merely a quasi-coherent sheaf that is locally free. As in classical case, these can also be thought as geometric vector bundles over $X$, and we will not make a distinction between these two viewpoints from now on. 

A \emph{quasi-projective} derived scheme is a derived scheme admitting a morphism $X \to \Proj^n$ such that it induces a locally closed embedding on truncation. As the underlying topological spaces of a derived scheme and its truncation coincide, this is the same as requiring the existence of maps $X \to U \to \Proj^n$, where the first morphism is a \emph{closed embedding} (meaning that it truncates to a closed embedding), and the second morphism is an open embedding. Choice of a locally closed embedding of $X$ to $\Proj^n$ determines a line bundle $\mathcal{O}(1)$ on $X$. Note that the notation $\mathcal{O}(1)$ is very abusive, as the line bundle is in no way independent of the chosen embedding. Nonetheless, it has some rather nice properties, of which the following proposition is an example
\begin{prop}\label{GlobalGeneration}
Let $E$ be a vector bundle on a derived scheme $X$ and $\mathcal{O}(1)$ a line bundle as above. For $n \gg 0$ the twisted vector bundle $E \otimes \mathcal{O}(n)$ is globally generated.
\end{prop}
A vector bundle $E$ is \emph{globally generated} if we can choose a morphism $\OO_X^{\oplus n} \to E$ from a free sheaf which becomes surjective on the truncation. This is the same as choosing global sections $s_1,...,s_n$ of $E$ so that the truncated global sections $\pi_0(s_1), ..., \pi_0(s_n)$ generate the classical vector bundle $\pi_0(E)$ on the truncation $\tau_0 X$. Another important fact we are going to need is the following:
\begin{prop}\label{KTheory}
Let $X$ be a quasi-projective derived scheme. Now $K^0(X)$ is naturally equivalent to the Grothendieck group of vector bundles, i.e., it is the Abelian group generated by equivalence classes of vector bundles modulo exact sequences.
\end{prop}
Note that the above proposition is not completely trivial, as $K^0(X)$ of a derived scheme is usually defined as the Abelian group generated by the equivalence classes of perfect complexes modulo cofibre sequences. Proving either of these two propositions would take us too much afield, and make this paper unnecessarily long, so the proofs are going to appear in \cite{An}. Proposition \fref{KTheory} follows quite easily from \fref{GlobalGeneration}, which in turn follows from a spectral sequence computation.

 third fact we are going to need is the following Proposition.
\begin{prop}\label{Grassmannian}
Let $X$ be a derived scheme. The space of maps from $X$ to the finite Grassmannian $\Gr(r,n)$ is equivalent to the space of surjections
$$\OO_X^{\oplus n} \to E$$
of vector bundles (in other words the map should truncate to a surjective morphism). More precisely, consider the subcategory $\mathcal{C}_X \subset \QCoh(X)_{\OO_X^{\oplus n}/}$ of quasi-coherent sheaves under $\OO_X^{\oplus n}$ so that the target is a vector bundle, the structure morphism is surjective, and so that the morphisms in $\mathcal{C}_X$ are equivalences of under $\OO_X^{\oplus n}$. Then $\mathcal{C}_X$ is an $\infty$-groupoid, or, in other words, a space. The above statement translates into the existence of a natural equivalence
$$\Hom(X, \Gr(r,n)) \simeq \mathcal{C}_X.$$
\end{prop}
\begin{proof}
One can just imitate a classical proof like the proof of Theorem 9.7.4 in the second edition of EGA I. The key point to keep in mind is that in derived algebraic geometry $\A^1$ represents the functor associating to $X$ the space of global sections of the structure sheaf $\OO_X$ (this is not true in spectral algebraic geometry, i.e., in the algebraic geometry modeled locally on $\mathbb{E}_\infty$-ring spectra, unless working over a field of characteristic 0).
\end{proof}

\begin{cor}\label{GlobalGenerationToGrassmannian}
Let $X$ be a derived scheme, $E$ a vector bundle of rank $r$ on $X$, and let $s_1,...,s_n$ be global sections of $E$ that generate it. Such a data induces a morphism $X \to \Gr(r,n)$ which is unique up to homotopy (i.e., they lie in the same path component in the mapping space).
\end{cor}

\section{Construction of $\Omega^*$}\label{BivariantConstruction}

This section is the heart of the paper. We extend the (derived) algebraic bordism $d\Omega_*$ groups to a full bivariant theory. Throughout the section we will work over a field $k$ of characteristic 0, and after the first subsection, we will be working in the (infinity-)category of quasi-projective derived $k$-schemes.

\subsection{Failure of the Proper-Smooth $\Omega^*$}
 
In the construction of algebraic cobordism, we begin with a construction which has pullbacks \emph{a priori} only for smooth morphisms. That is to say, smooth pullbacks can be defined on the level of cycles, and it is easy to see see that all the imposed relations are respected by these morphisms. On the other hand the construction of l.c.i. pullbacks more general is technical and messy business, taking most of the book \cite{LM} and the follow-up article \cite{Lev2} by Levine. The purpose of this section is to show that the most straightforward way of trying to make a bivariant version $\Omega^*$ of the construction of Levine and Morel can not produce a bivariant theory that has good properties. The properties we are interested in are:

\begin{enumerate}
\item the associated homology groups $\Omega^*(X \to pt)$ should be isomorphic to the algebraic bordism $\Omega_{-*}(X)$ of Levine-Morel;

\item we should have an orientation $\theta(f) \in \Omega^{-d}(X \stackrel{f}{\to} Y)$ for any l.c.i. morphism $f$ of relative dimension $d$; the induced Gysin morphisms
$$f^! = \theta(f) \bullet : \Omega_*(Y) \to \Omega_{*+d}(X)$$
should coincide with the l.c.i. pullbacks of Levine-Morel;

\item the above orientation should be strong (in the sense of Definition \fref{StrongDefn}) for smooth morphisms; as a special case, we obtain a natural Poincar\'e duality isomorphism
$$\bullet \theta(\pi_X): \Omega^*(X) \to \Omega_{d-*}(X)$$
for $\pi_X: X \to pt$ smooth of dimension $d$. 
\end{enumerate} 
While it is easy to give a geometric construction for a bivariant theory satisfying 1., we will see that the conditions 2. and 3. are not so easily satisfied.

We first roughly recall the construction of algebraic bordism groups $\Omega_*(X)$ by Levine and Morel given in \cite{LM}. Namely, we start with the free Abelian group generated by isomorphism classes of cycles of form
\begin{equation*}
[V \to X; \mathscr{L}_1, ..., \mathscr{L}_r]
\end{equation*}
where $V$ is a smooth scheme mapping properly to $X$ and $\mathscr{L}_i$ are line bundles on $V$ (the order of the line bundles does not matter), modulo the relation that disjoint union corresponds to sum. The purpose of the line bundles is to make applying the first Chern class operation a purely formal process. The algebraic bordism groups are then obtained by imposing relations to the above cobordism cycle groups.

It would therefore seem that the following strategy would be a natural one to follow:
\begin{enumerate}
\item[I.] start with Yokura's universal oriented bivariant theory $\mathbb{OM}^{\mathrm{proper}}_\mathrm{smooth}$ on quasi-projective schemes (see \cite{Yo1} Section 4) with proper morphisms as confined morphisms, smooth morphisms as specialized morphisms and all fibre squares as independent squares;

\item[II.] come up with the right bivariant analogues of the relations of Levine-Morel and verify that the quotient bivariant theory has the expected properties 1-3.
\end{enumerate}
We note that $\mathbb{OM}^{\mathrm{proper}}_\mathrm{smooth}(X \to Y)$ is the free Abelian group generated by cobordism cycles
$$[V \to X; \Li_1,...,\Li_r]$$
so that the map $V \to X$ is proper, and the composition $V \to Y$ is smooth, and $\Li_i$ are line bundles on $V$, and so it really is a natural bivariant extension of the cobordism cycle groups of Levine and Morel.

Such a construction is doomed to fail for a simple reason: if $X \to Y$ is a morphism which is not dominant, then no smooth morphism $V \to Y$ can factor through it, and therefore 
$$\mathbb{OM}^{\mathrm{proper}}_\mathrm{smooth}(X \to Y) \cong 0.$$
This means that we cannot equip any quotient theory of $\mathbb{OM}^{\mathrm{proper}}_\mathrm{smooth}$ with orientations along all l.c.i. morphisms. Suppose, for example, that $s: pt \to \Proj^1$ is a section of the structure map $\pi: \Proj^1 \to pt$. Now $\theta(s)$ is forced to be $0$, as it is the only element of $\mathbb{OM}^{\mathrm{proper}}_\mathrm{smooth}(pt \to \Proj^1)$, and hence 
$$\theta(\mathrm{Id}_{pt}) = \theta(s) \bullet \theta(\pi) = 0 \bullet \theta(\pi) = 0.$$
Even if one gave up the requirement that an orientation $\theta$ should be multiplicative, any $\theta$ would induce trivial Gysin pullbacks for all regular embeddings (of positive codimension), which is certainly false for the theory of Levine and Morel.

Even though constructing a good bivariant algebraic cobordism as a quotient of Yokura's universal theory $\mathbb{OM}^{\mathrm{proper}}_\mathrm{smooth}$ is hopeless, one might still hope to get a reasonable contravariant version of $\Omega_*$ as a quotient of the cohomology rings associated to $\mathbb{OM}^{\mathrm{proper}}_\mathrm{smooth}$.
Proving this seems to be highly nontrivial, and we close the subsection with the following related question:

\begin{que}
Let $X$ be a scheme. There is a natural map
$$\eta: \mathbb{OM}^{\mathrm{proper}}_\mathrm{smooth}(X \to X) \to \Omega^*(X \to X)$$
defined by formula
$$[V \to X; \Li_1,...,\Li_r] \mapsto c_1(\Li_1) \bullet \cdots \bullet c_1(\Li_r) \bullet [V \to X],$$
where the right hand side is the bivariant algebraic cobordism constructed in Section \fref{MainConstruction}, and $c_1(\Li_i) \in \Omega^1(X \to X)$ is the first Chern class of the line bundle $\Li_i$ as constructed in Section \fref{ChernClassSect} (see also Theorem \fref{ChernClasses}). Is $\eta$ surjective? Is there an easy description for the kernel of $\eta$?
\end{que}

We note that asking for the surjectivity of $\eta$ is essentially the same as asking that any quasi-smooth morphism $V \to X$ with a classical target $X$ is, up to cobordism, a global complete intersection of some smooth morphism $W \to X$.

\subsection{Derived algebraic bordism $d\Omega_*$ and its universality}\label{AlgebraicBordismSect}
The purpose of this section is to recall the construction of derived algebraic bordism $d\Omega_*$ from \cite{LS2}, and to provide it with a convenient universal property. We begin with the equivalence classes of proper morphisms $V \to X$ where $V$ is quasi-smooth. The related cycle is naturally denoted by $[V \to X]$. Arguably the most important relation is the \emph{homotopy fibre relation}. Namely, suppose we have a quasi-smooth $W$ mapping to $X \times \Proj^1$. In this case we identify the cobordism classes of the homotopy fibres $W \times^h_{X \times \Proj^1} (X \times 0)$ and $W \times^h_{X \times \Proj^1} (X \times \infty)$ with each other. 

Unfortunately the homotopy fibre relation is not enough, or this at least seems to be the case. The theory constructed above does have well defined functorial pullbacks along all quasi-smooth morphisms, and pushforwards for proper morphisms, and hence it does have first Chern classes. However, it is not clear that the Chern classes satisfy any formal group law, for instance. In order to remedy this, we use the old trick of tensoring first with the Lazard ring $\Laz$, and then enforcing the universal formal group law for globally generated line bundles on smooth schemes.

Alas, this still isn't enough! Suppose we have a simple normal crossing divisor $E$ on a smooth scheme $W$. Using the formal group law, we can express the class of $[E \to W]$ as a $\Laz$-linear combination of the prime divisors of $E$, their intersections, and Chern classes of line bundles on $W$. The latter expression can be lifted onto the support of $E$ to a class we are going to denote by $\eta_E \in \Omega_*(E)$. On the other hand we also have the fundamental class $1_E = [E \to E]$ on $E$. Even though their pushforwards from $E$ to $W$ agree, there is no reason why they would agree as elements of $\Omega_*(E)$ -- the formal group law doesn't really say anything about cycles of codimension 0. Hence we have to enforce that $1_E = \eta_E$. The purpose of these so called \emph{SNC relations} is to make sure that the two different ways of defining l.c.i. pullbacks (or rather, quasi-smooth pullbacks) coincide

We note here that last two of the families of relations are \emph{not} stable under pullback or pushforward. Hence, at each point, we actually get more relations than just the ones we initially wanted to enforce. 

The main result of \cite{LS2} (Theorem 5.12) is the following:
\begin{thm}[Algebraic Spivak Theorem]\label{AlgebraicSpivak}
For all quasi-projective derived schemes $X$ (over a field of characteristic 0) there is a natural isomorphism
\begin{equation*}
d\Omega_*(X) \cong \Omega_*(\tau_0 X),
\end{equation*}
where the right hand side is the algebraic cobordism, in the sense of Levine-Morel, of the truncation $\tau_0 X$ of $X$.
\end{thm}

Although not explicitly stated in \cite{LS2}, $d\Omega_*$ is also the universal \emph{oriented Borel-Moore homology} on quasi-projective derived schemes. As a slight modification of the concept introduced in \cite{LM} Chapter 5, these are functors $\B_*$ additive under disjoint unions and covariant for proper morphisms (proper pushforwards) that moreover come with Gysin pullbacks $f^!$ for all quasi-smooth morphisms $f$ of pure relative virtual dimension. There should also be a bilinear, associative and commutative exterior product
\begin{equation*}
\times: \B_*(X) \times \B_*(Y) \to \B_*(X \times Y),
\end{equation*}
and this product is required to have a unit element $1 \in \B_*(pt)$. For a line bundle $\Li$ on $X$ we may define a \emph{first Chern class operator} $c_1(\Li) \cap$ by $s^! s_*$, where $s$ is the zero section $X \to \Li$. These structures are required to satisfy the following properties
\begin{enumerate}
\item[($BM_1$)] Gysin pullbacks are functorial.

\item[($BM_2$)] For all homotopy Cartesian squares (as opposed for all transverse Cartesian squares)
\begin{center}
\begin{tikzpicture}[scale=2]
\node (A2) at (1,2) {$Y'$};
\node (B2) at (1,1) {$Y$};
\node (A1) at (0,2) {$X'$};
\node (B1) at (0,1) {$X$};
\path[every node/.style={font=\sffamily\small}]
(A1) edge[->] node[above]{$g'$} (A2)
(B1) edge[->] node[above]{$g$} (B2)
(A1) edge[->] node[left]{$f'$} (B1)
(A2) edge[->] node[right]{$f$} (B2)
;
\end{tikzpicture}
\end{center}
where $f$ is quasi-smooth and $g$ is proper, we have that $g'_*f'^! = f^! g_*$.

\item[($BM_3$)] The exterior product is compatible with pullbacks and pushforwards:
\begin{equation*}
(f \times g)_* (\alpha \times \beta) = f_* \alpha \times g_* \beta
\end{equation*}
and
\begin{equation*}
(f \times g)^! (\alpha \times \beta) = f^! \alpha \times g^! \beta.
\end{equation*}
whenever the operations are defined.

\item[($PB$)] Let $\pi: P \to X$ be a projective bundle of rank $r$, and denote by $\xi$ the operator $c_1(\OO_{P }(1)) \cap$. Now the map
\begin{equation*}
\bigoplus_{i=0}^{r} \B_*(X) \stackrel{\xi^i \pi^*}{\longrightarrow} \B_*(P)
\end{equation*}
is an isomorphism.
\item[($EH$)] Suppose $V \to X$ is a vector bundle and that $p: E \to X$ is a $V$-torsor. Now the map 
\begin{equation*}
p^!: \B_*(X) \to \B_*(E)
\end{equation*}
is an isomorphism.

\item[($CD$)] Suppose $W = (\Proj^N)^r$, $n_1,...,n_r$ non-negative integers, and let $p_i = W \to \Proj^N$ the $i^{th}$ projection. Let $p_1^*(X_N)^{n_1} \cdots p_r^*(X_N)^{n_r}$ be a global section of the line bundle $\OO(n_1) \boxtimes \cdots \boxtimes \OO(N_r)$, and let $E$ be its vanishing locus. Now the proper pushforward $\B_*(E) \to \B_*(W)$ is an injection. 
\end{enumerate} 
The exterior product $\times$ induces a commutative ring structure on $\B_*(pt)$ as well as $\B_*(pt)$-module structures on all other homology groups. All the operations are naturally $\B_*(pt)$-linear. For any quasi-smooth $X$ of pure virtual dimension, we may define the \emph{fundamental class} $1_X \in \B_*(X)$ as the Gysin pullback $\pi_X^!(1)$ of the identity element of $\B_*(pt)$, where $\pi_X: X \to pt$ is the quasi-smooth structure morphism. As one would expect, the projective bundle formula yields us a formal group law $F \in B_*(pt)[[x,y]]$ expressing the first Chern class of $c_1(\Li \otimes \Li') \cap$ in terms of $c_1(\Li)\cap$ and $c_1(\Li')\cap$. 

The main goal of this subsection is to provide the following universal property for $d\Omega_*$.

\begin{thm}\label{UniversalOBM}
Derived algebraic bordism $d\Omega_*$ is the universal oriented Borel-Moore homology theory in the following sense: given another such theory $\B_*$, there is a unique transformation $d\Omega_* \to \B_*$ commuting with pushforwards, pullbacks and exterior products, and sending $1 \in d\Omega_*(pt)$ to $1 \in \B_*(pt)$. Moreover, the map $\Laz = d\Omega_*(pt) \to \B_*(pt)$ classifies the formal group law of $\B_*$.
\end{thm}

It is easy to see that $d\Omega_*$ is indeed a Borel-Moore homology in the above sense. The properties $(BM_1) - (BM_3)$ follow directly from the definition, and the latter three properties follow from the algebraic Spivak's theorem \fref{AlgebraicSpivak} stating that $d\Omega_*(X) = \Omega_*(\tau_0 X)$ and from the corresponding claims for the underived bordism. 

\begin{proof}[Proof of Theorem \fref{UniversalOBM}]
As in \cite{LS2}, denote by $\mathcal{M}^+_*$ the theory that associates to each $X$ the free Abelian group on the isomorphism classes of proper maps $V \to X$ from (connected) quasi-smooth $V$. The pushforward is given by composition, pullback by homotopy fibre product and the exterior product by Cartesian product. If $\B_*$ is an oriented Borel-Moore homology theory, then we may define map $\psi: \mathcal{M}^+_* \to \B_*$ by sending $[f: V \to X]$ to $f_*\pi_V^!(1)$. It follows from the axioms $(BM_1) - (BM_3)$ that $\psi$ commutes with all the three operations. The proof is a straightforward exercise, and is left for the reader. We are left with the task of showing that the map $\mathcal{M}^+_* \to \B_*$ descends to a map $d\Omega_* \to \B_*$.

First we deal with the homotopy fibre relation. Suppose we have a proper map $W \to \Proj^1 \times X$, where $W$ is quasi-smooth, and let $W_0$ and $W_\infty$ be the two homotopy fibres mapping into $W$ by maps $i_0$ and $i_\infty$. Let $\Li$ be the pullback of $\OO(1)$ onto $W$. From $(BM_2)$ it follows that both $i_{0*} (1_{W_0})$ and $i_{\infty *}(1_{W_\infty})$ coincide with $c_1(\Li) \cap (1_W)$ in $\B_*$, so especially they must agree. Hence also the pushforwards of the fundamental classes of $W_0$ and $W_\infty$ to $\B_*(X)$ agree so the homotopy fibre relation holds on fundamental classes in $\B_*$, giving us a morphism $d\Omega_*^{\mathrm{naive}} \to \B_*$ (in the notation of \cite{LS2}).

In imposing the formal group law axiom we first tensor $d\Omega_*^{\mathrm{naive}}$ with the Lazard ring $\Laz$, and then artificially impose the universal formal group law on $\Laz \otimes_\Z d\Omega_*^{\mathrm{naive}}$. As the formal group law of $\B_*$ is classified by a map $\Laz \to \B_*(pt)$, we get a map  $\Laz \otimes_\Z d\Omega_*^{\mathrm{naive}} \to \B_*$ descending to a map $d \Omega_*^{\mathrm{pre}} \to \B_*$. As the SNC-axiom only mentions classical schemes, it holds in $\B_*$ by the arguments of \cite{LM} Chapter 7, and we get a map $d \Omega_* \to \B_*$.

The uniqueness of the map follows from the fact that we require the identity of the homology ring of the point to be preserved. The second claim follows from the fact that the map preserves first Chern classes.
\end{proof}

\subsection{Construction of the bivariant algebraic cobordism}\label{MainConstructionSect}
 
The functoriality of the derived algebraic bordism $d \Omega^*$ suggests us that the bivariant theory we are looking for should have orientations along all quasi-smooth morphisms, and pushforwards along all proper morphisms. Moreover, as quasi-smooth morphisms are stable under homotopy pullbacks, the construction of universal bivariant theory of Yokura, introduced earlier in Section \fref{BivariantTheories} should make sense. There is one problem however: derived schemes form an $\infty$-category, not an ordinary category. To circumvent the problem, we are going to work with the \emph{homotopy category} of quasi-projective derived schemes. This is the ordinary category whose objects are quasi-projective derived schemes, and so that the set of morphisms $X \to Y$ is the set of path components (also called \emph{equivalence classes} or \emph{homotopy classes}) of the mapping space $\map(X, Y)$ of derived schemes. Our confined morphisms will be proper morphisms, our independent squares will be homotopy Cartesian squares, and our specialized morphisms will be quasi-smooth morphisms. Note that as homotopy fibre products do not usually coincide with fibre products, Yokura's construction does not strictly speaking apply. However, extending the construction to this situation is a formality.

Let us denote by $\M^*$ Yokura's bivariant theory applied in the above situation. We recall that $\M^d(X \to Y)$ is the free Abelian group on equivalence classes of proper morphisms $[V \to X]$ so that the composition $V \to Y$ is quasi-smooth of relative virtual dimension $-d$, modulo the relation that disjoint union corresponds to summation.

\begin{rem}
By our earlier discussion in \fref{BivariantTheories}, any bivariant theory of this type has strong orientations along smooth morphisms. As a special case, Poincaré duality (in the sense of Proposition \fref{PoincareDuality}) is known to hold for smooth schemes.
\end{rem}

Before constructing the bivariant algebraic cobordism, we make the following useful definitions:

\begin{defn}\label{BivariantIdeal}
Suppose $\B^*$ is a bivariant theory on a category $\mathcal{C}$. A \emph{bivariant subset} $S$ consists of subsets $S(X \to Y) \subset \B^*(X \to Y)$ for all morphisms in $\mathcal{C}$. A bivariant subset $\I$ is a \emph{bivariant ideal} if
\begin{enumerate}
\item the subsets $\I(X \to Y) \subset \B^*(X \to Y)$ are subgroups;
\item the subsets $\I(X \to Y)$ are closed under bivariant pushforwards and pullbacks;
\item the bivariant product restricts to morphisms
$$\bullet: \B^*(X \to Y) \times \I(Y \to Z) \to \I(X \to Z)$$
and 
$$\bullet: \I(V \to X) \times \B^*(X \to Y) \to \I(V \to Y).$$
\end{enumerate}
Clearly, given a bivariant ideal $\I \subset \B^*$, the groups 
$$(\B^*/\I)(X \to Y) := \B^*(X \to Y)/\I(X \to Y)$$ 
inherit the structure of a bivariant theory.

Given a bivariant subset $S \subset \B^*$, we call the smallest bivariant ideal of $\B^*$ containing $S$ the bivariant ideal \emph{generated by $S$}, and denote it by $\langle S \rangle_{\B^*}$.
\end{defn}

It is now easy to make the main construction.

\begin{cons}[Bivariant Algebraic Cobordism]\label{MainConstruction}
We begin by observing that the induced homology groups $\M_*(X) := \M^{-*}(X \to pt)$ coincide with the cobordism cycle groups $\mathcal{M}_*^+(X)$ of \cite{LS2}. The construction of \emph{loc. cit.} also induces a natural surjection $\Theta: \Laz \otimes \mathcal{M}_*^+(X) \to d\Omega_*(X)$. Let us denote by $\M^*_\Laz$ the bivariant theory $\Laz \otimes \M^*$, and by $LS$ the bivariant subset of $\M^*_\Laz$ so that $LS(X \to pt)$ is naturally identified with the kernel of $\Theta$, and $LS(X \to Y)$ is empty whenever $Y$ is not the point $pt$. We now define the \emph{bivariant algebraic cobordism $\Omega^*$} as the quotient bivariant theory
$$\Omega^* := \M^*_\Laz / \langle LS \rangle_{\M^*_\Laz}.$$
We will call the associated cohomology ring $\Omega^*(X)$ the \emph{algebraic cobordism of $X$}.
\end{cons}

The reader who thinks that the above definition is unreasonable, should be comforted by the following proposition showing that the bivariant version of the homotopy fibre relation of Lowrey-Schürg is implied by the above relations.

\begin{prop}\label{BivariantHomotopyFibre}
Let $X \to Y$ be a morphism of quasi-projective derived schemes, and suppose $W \to \Proj^1 \times X$ is a proper morphism so that the composition $W \to \Proj^1 \times Y$ is quasi-smooth of relative dimension $d$. Let us denote by $W_0$ and $W_\infty$ the homotopy fibres of $W \to \Proj^1 \times X$ lying over $0 \times X$ and $\infty \times X$ respectively. We now have
$$[W_0 \to X] = [W_\infty \to X] \in \Omega^{-d}(X \to Y).$$
\end{prop}
\begin{proof}
By the construction of $\Omega^*$, we know this when $Y$ is the point $pt$. Applying the Poincaré duality \fref{StrongOrientations}, we can conclude that $[0 \to \Proj^1] = [\infty \to \Proj^1] \in \Omega^1(\Proj^1).$ Pulling this back along the composition $W \to \Proj^1 \times X \to \Proj^1$, we see that 
$$[W_0 \to W] = [W_\infty \to W] \in \Omega^1(W).$$ 
Moreover, as the morphism $f := W \to \Proj^1 \times Y \to Y$ is quasi-smooth of relative virtual dimension $d+1$, we can multiply the elements of the previous equation from the right by the orientation $\theta(f)$ in order to obtain
$$[W_0 \to W] = [W_\infty \to W] \in \Omega^{-d}(W \to Y).$$
Finally, as $W \to Y$ factors through the proper morphism $g: W \to X$ defined as the composition $W \to \Proj^1 \times X \to X$, we can apply the bivariant pushforward $g_*$ to the above equation to obtain the desired relation
$$[W_0 \to X] = [W_\infty \to X] \in \Omega^{-d}(X \to Y),$$
concluding the proof.
\end{proof}

Our next goal is to show that the associated homology groups $\Omega_*(X) := \Omega^{-*}(X \to pt)$ are naturally isomorphic to the derive algebraic bordism groups $d\Omega_*$ of Lowrey-Schürg (and hence to the algebraic bordism groups $\Omega_*(\tau_0 X)$ of Levine-Morel, so that there is no ambiguity in the notation for $X$ classical). In order to show this, we will have to study in more detail the bivariant ideal $\langle LS \rangle_{\M^*_\Laz} \subset \M^*_\Laz$ appearing in the Construction \fref{MainConstruction}. 

\begin{lem}
The elements of $\langle LS \rangle_{\M^*_\Laz}$ are linear combinations of elements of form 
\begin{equation*}
g_* (\alpha \bullet f^* (r) \bullet \beta)
\end{equation*}
where $r \in LS(W \to pt)$, $f$ is an arbitrary map to $pt$, $\alpha$ and $\beta$ are arbitrary bivariant elements, and $g$ is an arbitrary proper morphism (with the obvious restriction that the expression above makes sense).
\end{lem}
\begin{proof}
Clearly any element of the above form is in $\langle LS \rangle_{\M^*_\Laz}$. We are done if we can show that such elements form a bivariant ideal in the sense of Definition \fref{BivariantIdeal}. Stability under pushforward is trivial, and by bivariant axiom $A_{12}$
\begin{align*}
g_* (\alpha \bullet f^* (r) \bullet \beta) \bullet \gamma = g_* (\alpha \bullet f^* (r) \bullet (\beta \bullet \gamma)),
\end{align*}
showing stability under right multiplication. Suppose $g$ is a map $X \to Y$ and let $h$ be an arbitrary map $Y' \to Y$. By forming the homotopy Cartesian square
\begin{center}
\begin{tikzpicture}[scale=2]
\node (A2) at (1,2) {$Y'$};
\node (B2) at (1,1) {$Y$};
\node (A1) at (0,2) {$X'$};
\node (B1) at (0,1) {$X$};
\path[every node/.style={font=\sffamily\small}]
(A1) edge[->] node[above]{$g'$} (A2)
(B1) edge[->] node[above]{$g$} (B2)
(A1) edge[->] node[right]{$h'$} (B1)
(A2) edge[->] node[right]{$h$} (B2)
;
\end{tikzpicture}
\end{center}
and using the bivariant axioms $A_{13}$ and $A_{23}$ we see that
\begin{align*}
h^* g_* (\alpha \bullet f^* (r) \bullet \beta) &= g'_* h'^* (\alpha \bullet f^* (r) \bullet \beta) \\
&= g'_* ( h'''^*(\alpha) \bullet h''^* f^* (r) \bullet h'^*(\beta)) \\
&= g'_* ( \alpha' \bullet f'^*(r) \bullet \beta').
\end{align*}
Finally, supposing $\gamma \in \Laz \otimes \M^* (Y' \to Y)$, it follows from the bivariant axiom $A_{123}$ that 
\begin{align*}
\gamma \bullet g_* (\alpha \bullet f^* (r) \bullet \beta) = g'_*((g^*(\gamma) \bullet \alpha) \bullet f^*(r) \bullet \beta)
\end{align*}
finishing the proof.
\end{proof}

The proof of the following proposition is now straightforward. 

\begin{prop}\label{InducedHomologyIsBordism}
There is a natural isomorphism
\begin{equation*}
\Omega^*(X \to pt) \cong d\Omega_{-*}(X).
\end{equation*}
\end{prop}
\begin{proof}
By construction it is immediate that the only problem we might have is that $\Omega_*(X \to pt)$ has \emph{too many} relations. Recall that $\Omega_*(X \to pt)$ was defined to be the quotient of $\Laz \otimes \M^*$ by the relations $\langle LS \rangle (X \to pt)$, whereas $d \Omega_* (X)$ is the quotient of $\Omega_*(X \to pt)$ by $LS(X \to pt)$. We are therefore reduced to showing that $\langle LS \rangle (X \to pt) = LS(X \to pt)$. 

By construction $LS$ is stable under pushforwards and Gysin pullbacks. The latter means that $LS$ is stable under left multiplication by orientations of quasi-smooth morphisms, but as $\Laz \otimes \M^*$ is generated over $\Laz$ by pushforwards of orientations, we see that actually $LS$ is stable under left multiplication by \emph{arbitrary} bivariant elements. Hence, using the previous lemma, we need to show that elements of form 
\begin{equation*}
f^*(r) \bullet \beta
\end{equation*}  
are in $LS$, where $f$ is a map $Y \to pt$, $r \in LS(X \to pt)$ and $\beta \in \Laz \otimes \M^*(Y \to pt)$. But this element is, by definition, the bivariant external product $r \times \beta$  (see the proof of \fref{PoincareDuality}). As this product coincides with the homological exterior product, which is well defined in $d\Omega_*$, $LS$ must be stable under external products by arbitrary elements, finishing the proof. 
\end{proof}

Next we give a convenient universal property for $\Omega^*$. Let us first make the following definition.

\begin{defn}\label{BorelMooreBivariant}
We say that a bivariant theory (with proper pushforwards and orientations along quasi-smooth morphisms) is \emph{Borel-Moore} if the orientation is stable under pullbacks and its homology is an oriented Borel-Moore homology in the sense previously defined.
\end{defn} 

Bivariant algebraic cobordism is clearly an example of such a theory, and as one would expect, it is the universal one:

\begin{thm}
Bivariant algebraic cobordsim $\Omega^*$ is the universal Borel-Moore bivariant theory in the following sense: given another such a theory $\B^*$ there exists a unique Grothendieck transformation $\omega: \Omega^* \to \B^*$ respecting orientations. The restriction of this map to corresponding homology theories gives the universal morphism $d\Omega_* \to \B_*$. 
\end{thm}
\begin{proof}
By the universal property given in \cite{Yo1} there exists a unique Grothendieck transformation $\M^* \to \B^*$ respecting orientations. Using the universality of $d\Omega_*$ and the fact that $\B_*$ is an oriented Borel-Moore homology, we see that this map descends to a map $\Omega^* \to \B^*$. The last claim is trivial.
\end{proof}

Any Borel-Moore bivariant theory contains an oriented Borel-Moore homology theory, and hence it has a formal group law for first Chern classes (at this moment, we think them as operations on homology groups). We say that a Borel-Moore bivariant theory $\B^*$ is \emph{additive/multiplicative/periodic} if the corresponding formal group law has this property. Regard $\Z$ and $\Z [\beta, \beta^{-1}]$ as $\Laz$-algebras via the maps classifying the universal additive and the universal multiplicative periodic formal group laws respectively. The Borel-Moore bivariant theories $\Omega^*_{\mathrm{ad}} = \Omega^* \otimes_{\Laz} \Z$ and $\Omega^*_{\beta, \beta^{-1}} = \Omega^* \otimes_\Laz \Z[\beta, \beta^{-1}]$ are clearly the universal additive and the universal multiplicative and periodic Borel-Moore bivariant theories, respectively.  

\subsection{Chern classes in $\Omega^*$}\label{ChernClassSect}

Let $X$ be a quasi-projective derived scheme and let $E$ be a vector bundle of rank $r$ on $X$. The purpose of this section is to construct characteristic classes 
$$c_i(E) \in \Omega^i(X)$$
for $1 \leq i \leq r$ called the \emph{Chern classes} of the vector bundle $E$. We note that the Grothendieck approach to constructing Chern classes by invoking the so called splitting principle can not be applied for our cobordism rings $\Omega^*$, as the projective bundle formula is not known whenever $X$ is not smooth. The main results of this subsection make heavy use on the characteristic 0 assumption, as the Levine-Morel algebraic cobordism is not known to have any desirable properties in positive characteristic.

Recall that in the book \cite{LM} Levine and Morel construct Chern class operators $c_i(E) \cap: \Omega_*(X) \to \Omega_{*-i}(X)$ for any quasi-projective classical scheme $X$ and any vector bundle $E$ on $X$ in \cite{LM} Section 4.1.7. Hence, by identifying $c_i(E) \cap 1_X$ with an element of the cobordism ring $\Omega^i(X)$ using Poincaré duality \fref{PoincareDuality}, we obtain the desired Chern classes for smooth schemes. These classes satisfy the usual properties of Chern classes (see \cite{LM} Proposition 4.1.15).

\begin{rem}\label{AmbiguousFirstChernClass}
Recall that in Section \fref{AlgebraicBordismSect} we defined the first Chern class operator as $c_1(\Li) \cap = s^! s_*$, where $s: X \to Y$ is the zero section. However, when $X$ is classical, this definition coincides with the one given by Levine and Morel. This follows from \cite{LS2} Proposition 3.12 and \cite{LM} Lemma 7.4.1 (1) (see also \cite{LM} Definition 5.3.1).
\end{rem} 

The rest of the subsection is devoted to extending the theory of Chern classes to arbitrary quasi-projective derived schemes. We begin with globally generated vector bundles.

\begin{cons}[Chern classes for globally generated vector bundles]
Suppose $X$ is a quasi-projective derived $k$-scheme, and $E$ is a globally generated vector bundle of rank $r$. Let $s_1,...,s_n$ be global sections of $E$ generating the truncation $E$. This data gives rise to a map $f: X \to \Gr(r, n)$, unique up to homotopy (see Corollary \fref{GlobalGenerationToGrassmannian}), and $E$ is obtained by pulling back the universal quotient bundle $\mathcal{E}$ on $\Gr(r,n)$. As Grassmannians are smooth, they already have a good notion of Chern-classes, and hence we can define the Chern classes $c_i(E)$ of $E$ on $X$ to be the pullbacks $f^*(c_i(\mathcal{E}))$.
\end{cons}

The above construction is well defined.

\begin{lem}
Let $X$ be a quasi-projective derived scheme and let $E$ be a globally generated vector bundle on $X$. Now the Chern classes $c_i(E) \in \Omega^i(X)$ do not depend on the choice of global generators of $E$.
\end{lem}
\begin{proof}
Suppose we have two sequences $s_1,...,s_n$ and $s'_1,...,s'_{n'}$ of global generators of $E$, defining maps $f: X \to \Gr(r,n)$ and $g: X \to \Gr(r,n')$ respectively. We need to make sure that the Chern classes obtained by pulling back along $g$ agree with those pulled back along $f$. We note that $x_0 s_1,...,x_0s_n,x_1 s'_1,...,x_1 s'_{n'}$ are global sections of the vector bundle $\widetilde E (1)$ on $\Proj^1 \times X$, where $\widetilde E$ is the pullback of $E$. This defines a map $h: \Proj^1 \times X \to \Gr(r,n+n')$. Moreover, it is clear that the restrictions of the source factor as
\begin{align*}
0 \times X \stackrel g \to \Gr(r,n') \hookrightarrow \Gr(r,n+n')
\end{align*}
and
\begin{align*}
\infty \times X \stackrel f \to \Gr(r,n') \hookrightarrow \Gr(r,n+n').
\end{align*}
This shows that the Chern classes $c_i(\widetilde E(1))$ on $\Proj^1 \times X$ (defined via the map $h$) pull back to the Chern classes defined via the map $g$ when we pull back along $0 \times X \hookrightarrow \Proj^1 \times X$, and to the Chern classes defined via the map $f$ when we pull back along $\infty \times X \hookrightarrow \Proj^1 \times X$. By the bivariant homotopy fibre relation (see Proposition \fref{BivariantHomotopyFibre}), these classes agree, and hence the Chern classes of globally generated vector bundles are well defined.
\end{proof}

Extending to arbitrary vector bundles is now a matter of formal trickery. Denote by $F$ the universal formal group law, and by $F_-$ the associated universal difference group law. We remind the reader that these are ''inverses'' of each other: $F_-(F(x,y),y) = x$ and $F(F_-(x,y),y) = x$. Denote by $s_i$ the $i^{th}$ symmetric polynomial on $r$ variables $\mu_1,...,\mu_r$. We note that
\begin{equation*}
G^i(\mu_1,...,\mu_r,x) := s_i(F(\mu_1,x),...,F(\mu_r,x))
\end{equation*}
and 
\begin{equation*}
G_-^i(\mu_1,...,\mu_r,x) := s_i(F_-(\mu_1,x),...,F_-(\mu_r,x))
\end{equation*}
are $\Laz$ valued power series symmetric in the variables $\mu_i$, and hence have expressions $H^i(s_1,...,s_n,x)$ and $H^i_-(s_1,...,s_n,x)$ involving only $x$ and the elementary symmetric polynomials $s_i$. Now we have that 
\begin{equation*}
H^i(G_-^1(\mu_1,...,\mu_r,x),...,G_-^r(\mu_1,...,\mu_r,x),x) = s_i
\end{equation*}
and 
\begin{equation*}
H_-^i(G^1(\mu_1,...,\mu_r,x),...,G^r(\mu_1,...,\mu_r,x),x) = s_i,
\end{equation*}
as these amount to the same as
\begin{equation*}
s_i(F(F_-(\mu_1,x),x),...,F(F_-(\mu_r,x),x))
\end{equation*}
and 
\begin{equation*}
s_i(F_-(F(\mu_1,x),x),...,F_-(F(\mu_r,x),x)),
\end{equation*}
and because of the aforementioned property that $F$ and $F_-$ are inverses of each other. Similar argument shows that associativity results such as
\begin{equation*}
H^i(G^1(\mu_1,...,\mu_r,x),...,G^r(\mu_1,...,\mu_r,x),y) = G^i(\mu_1,...,\mu_r,F(x,y))
\end{equation*}
and
\begin{equation*}
H^i_-(G^1_-(\mu_1,...,\mu_r,x),...,G^r_-(\mu_1,...,\mu_r,x),y) = G^i_-(\mu_1,...,\mu_r,F(x,y))
\end{equation*}
hold.

Pulling back formulas over smooth varieties, we know that whenever $E$ is a globally generated vector bundle of rank $r$ on $X$, and $\Li$ is a globally generated line bundle, we have that $c_i(E \otimes \Li) = H^i(c_1(E),...,c_r(E),c_1(\Li))$. We are now ready to extend the definition of Chern classes to arbitrary vector bundles.

\begin{cons}[Chern classes for arbitrary vector bundles]
Suppose that $E$ is an arbitrary vector bundle on $X$. Now there exists a globally generated line bundle $\Li$ on $X$ such that the twist $E \otimes \Li$ is globally generated (see Proposition \fref{GlobalGeneration}). If we want to extend the Chern classes $c_i$ to all vector bundles in a way that respects the above twisting formulas, we are forced to set
\begin{equation*}
c_i(E) := H^i_-(c_1(E \otimes \Li),...,c_r(E \otimes \Li), c_1(\Li)).
\end{equation*}
\end{cons}

\begin{lem}
Let $X$ be a quasi-projective derived scheme and let $E$ be a vector bundle on $X$. Now the Chern classes $c_i(E) \in \Omega^{i}(X)$ do not depend on the choice of the globally generated line bundle $\Li$.
\end{lem}
\begin{proof}
Indeed, suppose that $\Li'$ is another globally generated line bundle such that $E \otimes \Li'$ is globally generated. Now $E \otimes \Li \otimes \Li'$ is globally generated, and by the associativity and inverse properties of $G^i$ and $G^i_-$ we conclude that the original definition of $c_i(E)$ in terms of $\Li$ agrees with
\begin{equation*}
H^i_-(c_1(E \otimes \Li \otimes \Li'),...,c_r(E \otimes \Li \otimes \Li'), c_1(\Li \otimes \Li')).
\end{equation*}
\end{proof}
 
We are now ready to verify expected properties of the Chern classes. As usual, the \emph{total Chern class} of a rank $r$ vector bundle is $c(E) := 1 + c_1(E) + \cdots + c_r(E)$, and the \emph{top Chern class} is $c_{\mathrm{top}}(E) := c_r(E)$.

\begin{thm}\label{ChernClasses}
The Chern classes in $\Omega^*$ defined above satisfy the following properties:
\begin{enumerate}[i.]
\item \emph{Contravariant functoriality:} $c_i(f^*E) = f^*(c_i(E))$.
\item \emph{Whitney sum formula:} whenever we have a short exact sequence
\begin{equation*}
0 \to E' \to E \to \overline E \to 0
\end{equation*}
of vector bundles on $X$, we have $c(E) = c(E') \bullet c(\overline E)$.
\item \emph{Normalization:} suppose $E$ is a vector bundle on $X$ with zero section $s$. Define the Euler class $e(E)$ of $E$ by $s^* s_!(1_X)$. Now $e(E) = c_\mathrm{top}(E)$.
\item \emph{Formal group law:} suppose $\Li$ and $\Li'$ are line bundles on $X$. Now
\begin{equation*}
c_1(\Li \otimes \Li') = F(c_1(\Li), c_1(\Li))
\end{equation*}
where $F$ is the universal formal group law over $\Laz$.
\end{enumerate}
\end{thm}
\begin{proof}
To prove $i.$, we first note that the Chern classes of globally generated vector bundles are natural in pullbacks. Indeed, suppose we have a globally generated vector bundle $E$ on $Y$, and a map $f: X \to Y$. Let $s_1,...,s_n$ be global sections generating $E$, and hence defining a map $Y \to \Gr$ to some Grassmannian. Now the data $(f^* E, f^*s_1,...,f^*s_n)$ corresponds to the map $X \stackrel f \to Y \to \Gr$, and by the functoriality of the bivariant pullback we can conclude that $c_i(f^* E) = f^* c_i(E)$. Extending this to arbitrary vector bundles is now a simple matter of choosing a globally generated line bundle $\Li$ on $Y$ such that $E \otimes \Li$ is globally generated and invoking the (inverse) twisting formula of the previous discussion.

For $ii.$, we note that if $E$ is the direct sum of $E'$ and $\overline E$, then we can pull back $E'$, $\overline E$ and $E$ from the same smooth variety, and hence we can pull back the relation $c(E')c(\overline E) = c(E)$ as well. In the general case, we note that the map $f: E \to \overline E$ defines the graph subbundle of $E \oplus \overline E$ locally consisting of elements $(e,f(e))$. Pulling back to $\Proj^1 \times X$ we get the subbundle of $E (1) \oplus \overline E(1)$ locally consisting of pairs $(x_0 e, x_1 f(e))$. Pulling back this bundle along $\infty \times X$ yields $E' \oplus \overline E$ whereas pulling back along $0 \times X$ gives $E$. By the homotopy fibre relation (Proposition \fref{BivariantHomotopyFibre}), we see that $c(E) = c(E' \oplus \overline E)$, proving $ii.$ The third and the fourth parts part can be trivially pulled back from smooth varieties where they are already known to hold, so we are done.
\end{proof}

We record here a rather trivial but useful fact:

\begin{thm}
Let $E$ be a vector bundle on a quasi-projective derived scheme $X$. Then the Chern classes $c_i(E)$ are nilpotent in $\Omega^*(X)$.
\end{thm}
\begin{proof}
This is known to hold for smooth schemes (after all, by the construction of Levine and Morel, $\Omega^i(X) \cong 0$ for $i > \dim(X)$). We are finished by noting that any vector bundle can be pulled back from a smooth scheme, and a ring homomorphism sends nilpotent elements to nilpotent elements.
\end{proof}

Finally, we show that the action of these newly defined cohomology Chern classes on the homology groups coincides with the action of original Chern classes in algebraic bordism $\Omega_*$.

\begin{prop}
Suppose $E$ is a vector bundle on a quasi-projective scheme $X$ and $\alpha \in \Omega_*(X)$. Now $c_i(E) \cap \alpha$ as defined in \cite{LM} Section 4.1.7 (denoted there by $c_i(E)(\alpha)$) agrees with $c_i(E) \bullet \alpha$.
\end{prop}
\begin{proof}
We know that there is a tower of projective bundles $\pi: P \to \cdots \to X$ such that the pullback $\pi^* E$ of $E$ has a filtration by vector bundles with line bundle quotients $\Li_1,...,\Li_r$. Moreover the pushforward $\pi_*: \Omega_*(P) \to \Omega_*(X)$ is surjective, and hence we can find $\widetilde \alpha$ mapping to $\alpha$. Using the push pull formula $c_i(E) \bullet \alpha = \pi_*(c_i(\pi^* E) \bullet \widetilde \alpha)$ following from the axioms of a bivariant theory, and using the Whitney sum formula of the previous proposition, we need only to check that the first Chern classes of line bundles agree. This is taken care of by the following general lemma.
\end{proof}

\begin{lem}
Let $\B^*$ be a bivariant theory which is $\mathcal{C}$-independent in the sense of \cite{Yo1} (meaning that all (homotopy) Cartesian squares obtained by pulling back a confined morphism are independent) with an orientation $\theta$,  and suppose $s: X \to Y$ is both confined and specialized. Suppose $\alpha \in \B_*(X)$. Then
$$s^!(s_*(\alpha)) = s^*(s_!(1_X)) \bullet \alpha \in \B_(X).$$
\end{lem}
\begin{proof}
Let us record here for later use the diagrams
\begin{center}
\begin{tikzpicture}[scale=2]
\node (A2) at (1,2) {$X$};
\node (B2) at (1,1) {$Y$};
\node (A1) at (0,2) {$W$};
\node (B1) at (0,1) {$X$};
\node (B3) at (2,1) {$pt$};
\path[every node/.style={font=\sffamily\small}]
(A1) edge[->] (A2)
(B1) edge[->] node[above]{$s$} (B2)
(A1) edge[->] node[left]{$s'$} (B1)
(A2) edge[->] node[left]{$s$} node[xshift=0.5cm,,circle,draw,inner sep=0pt, minimum size=0.5cm]{$\theta(s)$} (B2)
(B2) edge[->] (B3)
(B1) edge[->,bend right] node[yshift=-0.4cm,circle,draw,inner sep=0pt, minimum size=0.5cm]{$\alpha$} (B3)
;
\end{tikzpicture}
\end{center}
and
\begin{center}
\begin{tikzpicture}[scale=2]
\node (A2) at (1,2) {$X$};
\node (B2) at (1,1) {$Y$};
\node (C2) at (1,0) {$Y$};
\node (A1) at (0,2) {$W$};
\node (B1) at (0,1) {$X$};
\node (C1) at (0,0) {$X$};
\path[every node/.style={font=\sffamily\small}]
(A1) edge[->] (A2)
(B1) edge[->] node[above]{$s$} (B2)
(C1) edge[->] node[above]{$s$} (C2)
(A1) edge[->] node[left]{$s'$} (B1)
(B1) edge[->] node[left]{$\mathrm{Id_X}$} (C1)
(A2) edge[->] node[left]{$s$} (B2)
(B2) edge[->] node[left]{$\mathrm{Id_Y}$} (C2)
(A2) edge[->,bend left] node[xshift=0.5cm,circle,draw,inner sep=0pt, minimum size=0.5cm]{$\theta(s)$} (C2)
;
\end{tikzpicture}
\end{center}
where all squares are (homotopy) Cartesian. We can now compute that
\begin{align*}
s^!(s_*(\alpha)) &= \theta(s) \bullet s_*(\alpha) \\
&= s'_*\bigl(s^*(\theta(s)) \bullet \alpha \bigr) \quad (\text{$(A_{123})$ applied to the first diagram}) \\
&= s'_*\bigl(s^*(\theta(s)) \bigr) \bullet \alpha \quad (A_{13}) \\
&= s^* \bigl( s_* (\theta(s))\bigr) \bullet \alpha \quad (\text{$(A_{23})$ applied to the second diagram}) \\
&= s^* \bigl( s_* (1_X \bullet \theta(s))\bigr) \bullet \alpha \\
&= s^* \bigl( s_!(1_X) \bigr) \bullet \alpha,
\end{align*}
which proves the claim. 
\end{proof}

For general Borel-Moore bivariant theories, the results of this subsection can be summarized as follows:

\begin{cor}
Let $\B^*$ be a Borel-Moore bivariant theory in the sense of Definition \fref{BorelMooreBivariant}. Now the images of the Chern classes of $\Omega^*$ in the canonical map $\Omega^* \to \B^*$ give rise to contravariantly natural Chern classes in $\B^*$ satisfying normalization and Whitney sum formula. Moreover if $\Li$ and $\Li'$ are line bundles over $X$, then $c_1(\Li \otimes \Li') = F_{\B^*}(c_1(\Li), c_1(\Li'))$, where $F_{\B^*}$ is the formal group law of the oriented Borel-Moore homology theory $\B_*$. The action of these classes on the homology groups gives back the previously known Chern class operations.
\end{cor}
\begin{proof}
The claims about naturality, normalization and Whitney sum formulas are trivial. The formal group law follows from the fact that the ring map $\Omega^*(pt) \to \B^*(pt)$ is the one classifying $F_{\B^*}$ as there are literal identities $\B^*(pt) = \B_{-*}(pt)$ for all Borel-Moore bivariant theories. The final claim follows from the fact that the action of the Chern class operators commutes with transformations of oriented Borel-Moore homology theories.
\end{proof}

\section{Relations to other bivariant theories}

Let us denote by $K^0$ and $\op d \ch^*$ the bivariant algebraic $K$-theory, and the bivariant operational chow theory, extended to quasi-projective derived schemes in \cite{LS}. Both theories have proper morphisms as confined maps, quasi-smooth morphisms as specialized maps, and all homotopy Cartesian squares as independent squares. Although it is not noted in the paper, both orientations are stable under pullback. If $X \to Y$ is quasi-smooth (or more generally, perfect), then the corresponding orientation in $K^0(X \to Y)$ is $[\OO_X]$, which in a homotopy Cartesian square
\begin{center}
\begin{tikzpicture}[scale=2]
\node (A2) at (1,2) {$Y'$};
\node (B2) at (1,1) {$Y$};
\node (A1) at (0,2) {$X'$};
\node (B1) at (0,1) {$X$};
\path[every node/.style={font=\sffamily\small}]
(A1) edge[->] (A2)
(B1) edge[->] (B2)
(A1) edge[->] (B1)
(A2) edge[->] (B2)
;
\end{tikzpicture}
\end{center}
pulls back to $[\OO_{X'}]$ -- the orientation of $X' \to Y'$.

On the other hand, the orientation in $\op d \ch^*(X \to Y)$ is the collection of maps defined by the virtual pullback corresponding to the quasi-smooth morphism $X \to Y$ as constructed in \cite{Man}. Namely, in case of quasi-smooth embedding, we have a closed immersion $i: C_{\tau_0(X)}\tau_0(Y) \hookrightarrow \tau_0 (L_{X/Y}[1])$, where the right hand side is a vector bundle on $\tau_0 (X)$. Using the deformation to the normal cone construction, we do have a specialization map $d\ch_*(\tau_0(Y)) \to d \ch_*(C_{\tau_0(X)}\tau_0(Y))$. Composing the above specialization map with
\begin{equation*}
s^! i_*: d \ch_*(C_{\tau_0(X)}\tau_0(Y)) \to d \ch_*(\tau_0 (L_{X/Y}[1])) \to d\ch_*(\tau_0(X)) 
\end{equation*}
where the latter map $s$ is the zero-section, we obtain the \emph{virtual pullback}. As the cotangent complex is stable under homotopy pullbacks, we obtain in a natural way a \emph{refined virtual pullback} map for any homotopy Cartesian square as above, and this gives an operational class (\emph{loc. cit.}). For the same reasons it follows that the orientation is stable under bivariant pullbacks. These virtual pullbacks could also be constructed more directly using derived deformation to the normal cone, which we are going to use in the following subsection.

Now that we know that $\op d\ch^*$ and $K^0$ are Borel-Moore bivariant theories, we obtain canonical Grothendieck transformations $\Omega^* \to K^0$ and $\Omega^* \to \op d \ch^*$ preserving orientations. As it is often customary, instead of working with the algebraic $K$-theory \emph{per se}, we rather work with the \emph{Bott-periodic $K$-theory} $K^0 [\beta, \beta^{-1}]$, where $\beta$ is an element of cohomological degree $-1$. This is obtained from the bivariant theory $K^0$ by first tensoring it with $\Z[\beta, \beta^{-1}]$ over $\Z$, and then replacing the old orientations $\theta (f)$ with $\beta^d \theta(f)$, where $d$ is the relative virtual dimension of the quasi-smooth morphism $f$. The formal group law of this newly obtained theory is given by $F_{ K^0[\beta,\beta^{-1}]}(x,y)=x+y-\beta xy$; it is the universal multiplicative and periodic formal group law. The orientation is clearly stable under pullbacks, and hence we end up with another map $\Omega^* \to K^0[\beta,\beta^{-1}]$ of bivariant theories.

\subsection{The Cohomology Rings $\Omega^*$ and $K^0$.}\label{ConnerFloydSect}

The goal of this subsection is to show that the Bott periodic $K$-theory (and hence also the algebraic $K$-theory) cohomology ring is obtained from the algebraic cobordism by a simple change of coefficients. By the results of the previous section, the map
\begin{equation*}
E \mapsto r(E) + a_{11} c_1(E^\vee) \in \Omega^*(X),
\end{equation*}
where $r(E)$ denotes the rank of $E$, is additive in exact sequences of vector bundles, and hence this descends to a map $K^0(X) \to \Omega^*(X)$ (see Proposition \fref{KTheory}). Here $a_{11} \in \Laz$ is the coefficient of the term $xy$ in the universal formal group law. By tensoring, we obtain a \emph{Chern character} map 
\begin{equation*}
ch_\beta: K^0[\beta,\beta^{-1}](X) \to \Omega_{\beta,\beta^{-1}}^*(X) := \Omega^*(X) \otimes_\Laz \Z[\beta, \beta^{-1}]
\end{equation*}
where, as specified earlier, $\Z[\beta, \beta^{-1}]$ is to be regarded as an $\Laz$-algebra via the map classifying the universal multiplicative and periodic formal group law ($a_{11}$ is sent to $-\beta$, and all other variables are sent to 0). It is trivial that this map behaves well with respect to pullbacks. We claim that it also commutes with products and Gysin pushforwards. The product is easier of the two: namely,
\begin{align*}
r(E \otimes F) - \beta c_1(E^\vee \otimes F^\vee) &=  r(E) r(F) - r(F) \beta c_1(E^\vee) - r(E) \beta c_1(F^\vee) + \beta^2 c_1(E^\vee)c_1(F^\vee) \\
&= (r(E) - \beta c_1(E^\vee)) (r(F) - \beta c_1(F^\vee)),
\end{align*}
and therefore $ch(\alpha \bullet \beta) = ch(\alpha) \bullet ch(\beta)$. Recall that the ring structure on $K^0(X)$ induced by the bivariant product is the same as the usual one given by the derived tensor product, and the ring $K^0[\beta, \beta^{-1}](X)$ is just the ring of Laurent polynomials over $K^0(X)$. Before taking care of the pushforward, we need a lemma.

\begin{lem}
The Chern character map $ch: K^0[\beta,\beta^{-1}](X) \to \Omega_{\beta, \beta^{-1}}(X)$ preserves Chern classes.
\end{lem}
\begin{proof}
The first Chern class of a line bundle $\Li$ is given by $\beta^{-1}(1 - [\Li^\vee])$ in $K^0[\beta, \beta^{-1}]$, and its image in the Chern character map is
\begin{align*}
ch(\beta^{-1}(1 - [\Li^\vee])) &= \beta^{-1} (ch(1) - ch(\Li^\vee)) \\
&= \beta^{-1}(1 - (1 - \beta c_1(\Li))) \\
&= c_1(\Li).
\end{align*}
As $ch$ commutes with pullbacks, and as all vector bundles can be pulled back from a smooth scheme on which we can utilize the splitting principle, we see using the fact that $ch$ preserves multiplication, that the total Chern class of any vector bundle is preserved, which proves the claim.
\end{proof}

We are now going to prove that $ch$ commutes with Gysin pushforwards. The proof follows the classical approaches for proving Riemann-Roch type results using deformation to the normal bundle. We first prove the following toy case.

\begin{lem}\label{ToyModel1}
Let $E$ be a vector bundle on $X$. Now for the section $s: X \to \Proj := \Proj(1 \oplus E)$, the Gysin pushforward $s_!$ commutes with the Chern character map $ch$.
\end{lem}
\begin{proof}
We note that $X$ is cut out from $\Proj$ by a section to the vector bundle $\pi^* E \otimes \OO(1)$, where $\pi$ is the projection $\Proj \to X$, and therefore $s_!(1) = c_\mathrm{top}(\pi^* E \otimes \OO(1))$ in both theories. Hence, by the previous lemma, at least $ch(s_!(1_X)) = s_!(ch(1_X))$. But this is enough: in any bivariant theory $\B^*$ the ring map $s^*$ is surjective because $s$ is a section of $\pi$, and the Gysin pushforward map $s_!$ is $\B^*(\Proj)$-linear where the cohomology ring of $\Proj$ acts on the cohomology of $X$ via the pullback $s^*$ (see for example \cite{FM} Section 2.5).
\end{proof}

\begin{rem}\label{ToyModel2}
We note that in the previous lemma all that is actually needed is that $X$ is a quasi-smooth retract of $P$ that can be expressed as the derived vanishing locus of a section of a vector bundle $E$ on $P$ (and hence $[X \to P]$ corresponds to the top Chern class of $E$).
\end{rem}

\begin{lem}\label{ChernCharacterCommutesWithPushforward}
Let $f: X \to Y$ be a proper quasi-smooth map, and $\alpha \in K^0[\beta, \beta^{-1}](X)$. Now $ch$ commutes with the Gysin pushforward $f_!$:
\begin{equation*}
ch(f_!(\alpha)) = f_!(ch(\alpha)).
\end{equation*} 
\end{lem}

As we are working with quasi-projective schemes, the map $f$ is projective. Using the functoriality of Gysin pushforward, we have two cases to consider: either $f$ is the natural projection $\pi_Y: \Proj^n \times Y \to Y$ or then $f$ is a quasi-smooth embedding $j: X \hookrightarrow Y$.

\begin{proof}[Proof of Lemma \fref{ChernCharacterCommutesWithPushforward} for the projection $\pi_Y: \Proj^n \times Y \to Y$] Let us denote by $\pi$ the structure morphism $\Proj^n \to pt$. As the projective bundle formula holds for the algebraic $K$-theory rings of derived schemes (which follows from Theorem 3.16 of \cite{KST} and the fact that 
$[\OO_{\Proj^{n-i} \times Y}] = ([\OO_{\Proj^n \times Y}] - [\OO_{\Proj^n \times Y}(-1)])^i \in K^0(\Proj^n \times Y)$), we may assume that we have an element of form 
$$[\OO_{\Proj^i}] \times \alpha \in K^0[\beta, \beta^{-1}](\Proj^n \times Y),$$
where $\alpha$ is an element of $K^0[\beta, \beta^{-1}](Y)$, $\Proj^i$ is linearly embedded in $\Proj^n$, and $\times$ is the bivariant exterior product (see \cite{FM} Part I Section 2.4). By the standard properties of the bivariant exterior product, we see that
\begin{align*}
\pi_{Y*} ([\OO_{\Proj^i}] \times \alpha) &= \pi_*([\OO_{\Proj^i}]) \times  \mathrm{Id}_{Y*}(\alpha) \\
&= \pi_*([\OO_{\Proj^i}]) \times \alpha
\end{align*}
so the question reduces to the case of $Y = pt$. But now we are dealing with only smooth varieties, and therefore the claim follows from results in algebraic bordism by Poincaré duality.
\end{proof}

In order to prove the remaining case, we will employ the strategy of the proof of Theorem 15.2 from \cite{Ful}. If $X \to Y$ is a quasi-smooth closed embedding, then we define the \emph{normal bundle} $\mathcal{N}_{X/Y}$ to be the vector bundle $(L_{X/Y}[1])^\vee$ (the cotangent complex of a quasi-smooth embedding is known to be the homological shift $E[-1]$ of a vector bundle $E$, which is called the \emph{conormal bundle}).

\begin{proof}[Proof of Lemma \fref{ChernCharacterCommutesWithPushforward} for quasi-smooth closed embedding] Let $j: X \hookrightarrow Y$ be a quasi-smooth closed embedding. Denote by $\bl_{0 \times X} (\Proj^1 \times Y)$ the derived blow up of $Y$ at $0 \times X$ (as constructed in \cite{Khan}). Consider now the following diagram:
\begin{center}
\begin{tikzpicture}[scale=2]
\node (X1) at (0,2) {$0 \times X$};
\node (X2) at (0,1) {$\Proj^1 \times X$};
\node (X3) at (0,0) {$\infty \times X$};
\node (E) at (1.35,2) {$\Proj(\OO_X \oplus \mathcal{N}_{X/Y})$};
\node (plus) at (2.2,2) {$+$};
\node (Y1) at (2.65,2) {$\bl_X (Y)$};
\node (Y2) at (2,1) {$\bl_{0 \times X} (\Proj^1 \times Y)$};
\node (Y3) at (2,0) {$Y$};
\node (end1) at (4,2) {$0$};
\node (end2) at (4,1) {$\Proj^1$};
\node (end3) at (4,0) {$\infty$};
\path[every node/.style={font=\sffamily\small}]
(X1) edge[->] node[above]{$s$} (E)
(X2) edge[->] node[above]{$F$} (Y2)
(X3) edge[->] node[above]{$j$} (Y3)
(Y1) edge[->] (end1)
(Y2) edge[->] node[above]{$\pi$} (end2)
(Y3) edge[->] (end3)
(X1) edge[->] (X2)
(X3) edge[->] (X2)
(Y1) edge[->] node[right]{$i_0$} (Y2)
(E) edge[->] node[right]{$\iota$} (Y2)
(Y3) edge[->] node[right]{$i_\infty$} (Y2)
(end1) edge[->] (end2)
(end3) edge[->] (end2)
;
\end{tikzpicture}
\end{center}
where $F$ is the strict transformation of $\Proj^1 \to X \to \Proj^1 \times Y$ (existence guaranteed by Lemma \fref{DerivedBlowUp}), $s$ is the map from $0 \times X$ to the exceptional divisor (by Lemma \fref{DerivedBlowUp} (c) the image of $F$ does not meet the image of $i_0$), $\iota$ and $i_0$ are the inclusions of the exceptional divisor and the strict transform of $0 \times Y \to \Proj^1 \times Y$ respectively, and their sum is naturally identified with the homotopy fibre of $\pi$ over $0$. Note also that $s$ is a section of the projective bundle $\Proj(\OO_X \oplus \mathcal{N}_{X/Y})$.

Note that it follows from the above diagram that for either of the two Borel-Moore bivariant theories we are interested in, we have
\begin{equation}\label{GRREq1}
i_{\infty !}(1) = \iota_!(1) + i_{0!}(1) - \beta \iota_!(1) \bullet i_{0!}(1) \in \B^*(\bl_{0 \times X} (\Proj^1 \times Y)),
\end{equation}
where we have denoted abusively by $1$ the fundamental classes of various derived schemes in the diagram. 

Let $E$ be a vector bundle on $X$, and denote by $\widetilde E$ the pullback of $E$ to $\Proj^1 \times X$. It follows from the push-pull formula applied to the bottom left square (which is homotopy Cartesian), that $i_\infty^* (F_* \widetilde E) \simeq j_*(E)$), and therefore
\begin{align*}
i_{\infty !}(ch([j_* E])) &= i_{\infty !}(i_\infty^*(ch([F_* \widetilde E]))) \\
&= ch([F_* \widetilde E]) \bullet i_{\infty!}(1_Y). \quad (\text{projection formula})
\end{align*}
It now follows from the equation (\fref{GRREq1}) above that 
\begin{align*}
ch([F_* \widetilde E]) \bullet i_{\infty!}(1_Y) &= ch([F_* \widetilde E]) \bullet (\iota_!(1) + i_{0!}(1) - \beta \iota_!(1) \bullet i_{0!}(1)) 
\end{align*}
and from the fact that the images of $F$ and $i_0$ do not meet we can conclude that 
\begin{equation}\label{GRREq2}
i_{\infty !}(ch([j_* E])) = ch([F_* \widetilde E]) \bullet \iota_!(1).
\end{equation}
We can now apply projection formula again to obtain
\begin{align}\label{GRREq3}
ch([F_* \widetilde E]) \bullet \iota_!(1) &= \iota_!(\iota^*(ch([F_* \widetilde E]))) \\
&= \iota_! (ch([s_*E]))
\end{align}
where the last equation follows from the application of push-pull formula to the top left square (which is homotopy Cartesian). We can now use Lemma \fref{ToyModel1} to conclude that
\begin{equation}\label{GRREq4}
\iota_!(ch([s_*E])) = \iota_!(s_!(ch(E))).
\end{equation}

Denote by $q$ the composition $\bl_{0 \times X} (\Proj^1 \times Y) \to \Proj^1 \times Y \to Y$. We have shown that
\begin{align*}
ch([j_* E]) &= 	q_!(i_{\infty !} (ch[j_* E])) \\
&= q_!(\iota_!(s_!(ch(E)))) \quad (\text{equations $(2)$-$(5)$}) \\
&= j_!(ch(E)) \quad (j \simeq q \circ \iota \circ s),
\end{align*}
which is exactly what we wanted. 
\end{proof}

We are left with proving the claims corresponding strict transforms in derived blow ups.

\begin{lem}\label{DerivedBlowUp}
Let $Z \hookrightarrow X \hookrightarrow Y$ be a chain of quasi-smooth embeddings. Now 
\begin{enumerate}[(a)]
\item There is a natural map $\bl_Z(X) \to \bl_Z(Y)$ called the \emph{strict transform} of $X$.
\item This map is a quasi-smooth embedding.
\item Given two quasi-smooth embeddings $X_1 \hookrightarrow Y$ and $X_2 \hookrightarrow Y$, then the strict transforms of $X_1$ and $X_2$ do not meet in $\bl_{X_1 \times^h_Y X_2}(Y)$.
\end{enumerate}
\end{lem}
\begin{proof}
Denote by $E_X$ the exceptional divisor of $\bl_Z(X)$, and by $E_Y$ the exceptional divisor of $\bl_Z(Y)$. The diagram
\begin{center}
\begin{tikzpicture}[scale=2]
\node (A2) at (1,2) {$\bl_Z(X)$};
\node (B2) at (1,1) {$Y$};
\node (A1) at (0,2) {$E_X$};
\node (B1) at (0,1) {$Z$};
\path[every node/.style={font=\sffamily\small}]
(A1) edge[->] (A2)
(B1) edge[->] (B2)
(A1) edge[->] (B1)
(A2) edge[->] (B2)
;
\end{tikzpicture}
\end{center}
is a \emph{virtual Cartier divisor} lying over $(Y,Z)$ as in \cite{Khan}. Hence there is a unique morphism $\bl_Z(X) \to \bl_Z(Y)$ pulling back $E_Y$ to $E_X$. This takes care of $a)$.

The second part is local, and we may therefore use the local construction for the derived pullback given in \emph{loc. cit.} Namely, locally, we have the following homotopy fibre diagram
\begin{center}
\begin{tikzpicture}[scale=2]
\node (A0) at (-1,2) {$Z$};
\node (B0) at (-1,1) {$pt$};
\node (A2) at (1,2) {$Y$};
\node (B2) at (1,1) {$\A^n$};
\node (A1) at (0,2) {$X$};
\node (B1) at (0,1) {$\A^m$};
\path[every node/.style={font=\sffamily\small}]
(A0) edge[right hook->] (A1)
(B0) edge[right hook->] (B1)
(A0) edge[->] (B0)
(A1) edge[right hook->] (A2)
(B1) edge[right hook->] (B2)
(A1) edge[->] (B1)
(A2) edge[->] (B2)
;
\end{tikzpicture}
\end{center}
and the derived pullbacks are the homotopy pullbacks of the classical blow ups on the affine spaces. As the natural map $\bl_0(\A^m) \to \bl_0(\A^n)$ is a regular embedding, we are done with $b)$.

The question is again local. Let $d_i$ be the virtual codimension of $X_i \hookrightarrow Y$. We may assume the existence of homotopy pullback diagrams
\begin{center}
\begin{tikzpicture}[scale=2]
\node (A2) at (1,2) {$Y$};
\node (B2) at (1,1) {$\A^{d_i}$};
\node (A1) at (0,2) {$X_i$};
\node (B1) at (0,1) {$pt$};
\path[every node/.style={font=\sffamily\small}]
(A1) edge[->] (A2)
(B1) edge[->] (B2)
(A1) edge[->] (B1)
(A2) edge[->] (B2)
;
\end{tikzpicture}
\end{center}
giving us the diagram
\begin{center}
\begin{tikzpicture}[scale=2]
\node (A2) at (1,2) {$Y$};
\node (B2) at (1,1) {$\A^{d_1 + d_2}$};
\node (A1) at (0,2) {$X_1 \cap X_2$};
\node (B1) at (0,1) {$pt$};
\path[every node/.style={font=\sffamily\small}]
(A1) edge[->] (A2)
(B1) edge[->] (B2)
(A1) edge[->] (B1)
(A2) edge[->] (B2)
;
\end{tikzpicture}
\end{center}
Using again the fact that the derived blow up is given by homotopy pullback of the classical blow up on affine spaces, we conclude that the strict transforms of $X_i$ are disjoint because this holds for the strict transforms of $\A^{d_1}$ and $\A^{d_2}$ in $\bl_0(\A^{d_1 + d_2})$. This takes care of $c)$.
\end{proof}

We are now ready to prove the main theorem of the section: the cohomology groups $\Omega_{\beta,\beta^{-1}}^*(X)$ and $K^0[\beta, \beta^{-1}](X)$ are isomorphic for all $X$.

\begin{thm}\label{KTheoryCobordism}
Let $X$ be a quasi-projective derived scheme. The natural map $\omega_{K}: \Omega_{\beta,\beta^{-1}}^*(X) \to K^0[\beta, \beta^{-1}](X)$ induced by the universal Grothendieck transformation from $\Omega^*$ is an isomorphism.
\end{thm} 
\begin{proof}
We claim that the inverse to $\omega_K$ is given by $ch$. First of all, by construction, both $\omega_K$ and $ch$ preserve identity elements, as can be checked by the simple computations
\begin{align*}
\omega_K(1_X) = [\OO_X]
\end{align*}
and 
\begin{align*}
ch([\OO_X]) &= 1_X - \beta c_1(\OO_X^\vee) \\
&= 1_X.
\end{align*}

As $ch$ commutes with pullbacks by Theorem \fref{ChernClasses}, and by Gysin pushforwards by Lemma \fref{ChernCharacterCommutesWithPushforward}, we see that the transformation
$$ch \circ \omega_K: \Omega_{\beta,\beta^{-1}}^* \to \Omega_{\beta,\beta^{-1}}^*$$
commutes with pullbacks, Gysin pushforwards, and preserves the identity elements. But there is clearly only one such transformation --- the identity --- as $\Omega^*(X)$ is generated by classes of form $f_!(1_V)$, where $f: V \to X$ is both proper and quasi-smooth.

In order to show that also $\omega_K \circ ch$ is the identity, we are left to show that the transformation $\omega_K$ is surjective. But this is easy: as the first Chern class of a vector bundle $E$ in $K^0(X)$ is $\beta (r(E) - [E^\vee])$ (the usual $K$-theory Chern class twisted by $\beta$), and as $\omega_K$ preserves Chern classes, we see that
\begin{align*}
\omega_K (r(E) 1_X - \beta^{-1} c_1(E^\vee)) &= r(E) - \beta^{-1} \beta (r(E^\vee) - [E]) \\
&= [E] 
\end{align*}
proving the surjectivity of $\omega_K$, and finishing the proof.
\end{proof}

\begin{rem}
Setting $\beta=1$ gives the ordinary algebraic $K^0$-theory from the Bott periodical one, and hence also this can be obtained from the algebraic cobordism ring.
\end{rem}

\begin{rem}\label{HomotopyInvarianceFails}
The above theorem has multiple interesting consequences for the bivariant cobordism theory $\Omega^*$. First of all, we know that there exists toric varieties which have very large $K^0$-groups: in fact, they are uncountable \cite{Gub}. Hence the algebraic cobordism rings of these toric varieties are uncountable as well. This can never happen for the operational cobordism rings: they are always countable for toric varieties. Hence $\Omega^* \neq \op \Omega^*$.

It is not true in general that when we have an affine space bundle $p: E \to X$ (or even a trivial vector bundle), then $p^*: K^0(X) \to K^0(E)$ is an isomorphism. Nor does it hold in general that when $i: U \hookrightarrow X$ is an open embedding, then $i^*: K^0(X) \to K^0(U)$ is a surjection. Both of these claims are therefore false for algebraic cobordism rings as well. The failure of the homotopy property should set our $\Omega^*$ apart from any $\A^1$-homotopical bivariant theory. In particular, the groups $\Omega^*(X \to Y)$ constructed here cannot always agree with the groups $MGL^{BM}_{2*,*}(X/Y)$ constructed by Déglise in \cite{De}.
\end{rem}

\subsection{A Bivariant Intersection Theory}\label{BivariantChowSect}

Using the new theory of bivariant algebraic cobordism, we may construct a new bivariant intersection theory. Namely, consider $\Z$ as a $\Laz$-algebra along the map classifying the additive formal group law. We define $\ch^*$ as the universal additive Borel-Moore bivariant theory $\Z \otimes_{\Laz} \Omega^*$. We can combine the Proposition \fref{InducedHomologyIsBordism} with algebraic Spivak's theorem \fref{AlgebraicSpivak} in order to see that $\ch^*(X \to pt)$ is the Chow group $\ch_*(\tau_0 X)$ of $X$ for all quasi-projective derived schemes, and it follows from Poincaré duality \fref{PoincareDuality} that for smooth $X$ the induced cohomology ring $\ch^*(X)$ agrees with the usual Chow ring.

The next theorem is an extension of some previously known results, and its proof uses the old idea of Quillen of twisting cohomology theories. However, we will first need some definitions.

\begin{defn}
Let $\B$ be a Borel-Moore bivariant theory, and suppose $\tau = (b_0, b_1, b_2,...)$ is an infinite sequence of elements of $\B(pt)$ with $b_0 = 1$. For any line bundle $\Li$ on a quasi-projective derived scheme $X$, define its \emph{inverse Todd class} by
\begin{equation*}
\mathrm{Td}^{-1}_{\tau}(\Li) :=  \sum_i b_i c_1(\Li)^i \in \B^*(X).
\end{equation*}
This formula has a natural extension for all vector bundles so that
$$\mathrm{Td}^{-1}_{\tau}(E) = \mathrm{Td}^{-1}_{\tau}(E') \bullet \mathrm{Td}^{-1}_{\tau}(E'')$$
whenever we have a short exact sequence
$$0 \to E' \to E \to E'' \to 0$$
of vector bundles on $X$, which can be derived via the usual way of writing the symmetric power series
$$\prod_{j=1}^r \sum_i b_i x_j^i \in \B^*(pt)[[x_1,...,x_r]]$$
in terms of the elementary symmetric polynomials to obtain a power series $\mathrm{Td}^{-1}_{\tau,r}(s_1,...,s_r) \in \B^*(pt)[[s_1,...,s_r]]$, and then setting
$$\mathrm{Td}^{-1}_{\tau}(E) := \mathrm{Td}^{-1}_{\tau,r}(c_1(E), ..., c_r(E)) \in \B^*(X)$$
for a vector bundle $E$ of rank $r$. For any vector bundle $E$ on $X$, we define the \emph{Todd class} $\mathrm{Td}_\tau(E)$ as the inverse $(\mathrm{Td}^{-1}_\tau(E))^{-1}$ of the inverse Todd class.

If $f$ is a quasi-smooth map, with cotangent complex equivalent to $E_1 \to E_0$, we define its inverse Todd class as
\begin{equation*}
\mathrm{Td}^{-1}_{\tau}(f) := {\mathrm{Td}^{-1}_{\tau}(E_0^\vee) \over \mathrm{Td}^{-1}_{\tau}(E_1^\vee)}.
\end{equation*}
This is well defined, as the multiplicativity of $\mathrm{Td}^{-1}_{\tau}$ under exact sequences implies that $\mathrm{Td}^{-1}_{\tau}$ is well defined for $K$-theory classes of perfect complexes (see Proposition \fref{KTheory}). Properties of the cotangent complex show that the inverse Todd class is multiplicative in compositions. Hence, we may equip the underlying bivariant theory of $\B$ with another orientation
\begin{equation*}
\theta_\tau(f) := \mathrm{Td}^{-1}_\tau(f) \bullet \theta(f) 
\end{equation*}
giving us a new Borel-Moore bivariant theory $\B_\tau$. We call this process of modifying Borel-Moore bivariant theories \emph{twisting}.
\end{defn}

\begin{lem}\label{TwistedChern}
Let $\B^*$ be a Borel-Moore bivariant theory, and let $\tau = (1, b_1, b_2, ...)$ be a sequence of elements of $\B^*(pt)$. Now the first Chern classes $c_1^\tau$ of the twisted bivariant theory $\B^*_\tau$ satisfy the equation
\begin{align*}
c_1^\tau(\Li) &= \mathrm{Td}_\tau(\Li^\vee) \bullet c_1(\Li) \\
&= c_1(\Li) + b_1 c_1(\Li)^2 + b_2 c_1(\Li)^3 + \cdots \in \B^*(X)
\end{align*}
for all quasi-projective derived schemes $X$, and all line bundles $\Li$ on $X$.
\end{lem}
\begin{proof}
Let us denote by $f_!^\tau$ the \emph{twisted Gysin pushforward} for a proper quasi-smooth morphism $f$ of pure relative dimension. Let $\Li$ be a vector bundle on $X$, and denote by $s$ the zero section $X \to \Li$, and by $\pi$ the structure morphism $\Li \to X$. Recalling that the cotangent complex of $X \to \Li$ is quasi-isomorphic to $\Li^\vee \to 0$, we can compute that
\begin{align*}
s^*(s_!^\tau(1_X)) &= s^*(s_*(1_X \bullet \mathrm{Td}^{-1}_\tau (L_{X / \Li}) \bullet \theta(s))) \\
&= s^*(s_*(1_X \bullet \mathrm{Td}_\tau (\Li^\vee) \bullet \theta(s))) \\
&= s^*(s_*(\mathrm{Td}_\tau(\Li^\vee) \bullet 1_X \bullet \theta(s))) \\
&= s^*(\mathrm{Td}_\tau(\pi^* \Li^\vee) \bullet s_*(1_X \bullet \theta(s))) \quad (\text{projection formula}) \\
&= s^*(\mathrm{Td}_\tau(\pi^* \Li^\vee) \bullet s_!(1_X)) \\
&= \mathrm{Td}_\tau(\Li^\vee) \bullet s^*(s_!(1_X)),
\end{align*}
which proves the claim.
\end{proof}

\begin{thm}
With $\Q$-coefficients, we have a natural isomorphism of bivariant theories.
\begin{equation*}
\Omega^*_\Q \cong \Laz_\Q \otimes \ch^*_\Q 
\end{equation*}
\end{thm}
\begin{rem}
$\B^*_\Q$ is a standard notation for $\Q \otimes \B^*$.
\end{rem}
\begin{proof}
It follows from well known principles that a Borel-Moore bivariant theory $\B^*$ with rational coefficients can be twisted to a Borel-Moore bivariant theory $\B^*_\tau$ the additive (and hence any other) formal group law over $\B^*(pt)$. This follows from the above Lemma \fref{TwistedChern} and \cite{LM} Lemma 4.1.29 (the part of the power series in the statement is played by the second expression for $c_1^\tau$ in Lemma \fref{TwistedChern}).

Hence, we may twist $\Omega^*_\Q$ to have the additive formal group law so that we get a canonical map $\Laz_\Q \otimes \ch^*_\Q \to \Omega^*_\Q$. Similarly, we can use the inverse of the previous twist in order to make $\Laz_\Q \otimes \ch^*_\Q$ to have the universal formal group law, and this gives us a canonical map $\Omega^*_\Q \to \Laz_\Q \otimes \ch^*_\Q$. We claim that these maps are inverses of each other.

We first note that the composites $\Omega^*_\Q  \to \Omega^*_\Q$ and $\Laz_\Q \otimes \ch^*_\Q \to \Laz_\Q \otimes \ch^*_\Q$ at least preserve orientations, and are therefore at least endomorphisms of Borel-Moore bivariant theories. The former morphism is the identity, as that is the only endomorphism of $\Omega^*_\Q$ as a Borel-Moore bivariant theory. Moreover, as any endomorphism of $\Laz_\Q \otimes \ch^*_\Q$ preserves the image of the universal map $\eta: \ch^*_\Q \to \Laz_\Q \otimes \ch^*_\Q$ of additive bivariant Borel-Moore theories, and as as the above endomorphism induces identity on $\Laz_\Q = \Laz_\Q \otimes \ch^*_\Q(pt),$ the given endomorphism of $\Laz_\Q \otimes \ch^*_\Q$ must be the identity at least on the $\Laz$-span of the image of $\eta$, which in this case is the whole theory.
\end{proof}

A corollary of the results of the previous section and the above theorem is a Riemann-Roch transformation between $K^0$ and $\ch^*_\Q$.

\begin{thm}\label{DerivedGRR}
There is a natural transformation of cohomology rings
\begin{equation*}
ch: K^0(X) \to \ch^*_\Q(X) 
\end{equation*}
for all quasi-projective derived schemes $X$. Moreover, after tensoring with $\Q$ this becomes an isomorphism.
\end{thm}
\begin{proof}
Let $\Z_a$ and $\Z_m$ be the integers considered as $\Laz$-algebras via the map classifying the formal group laws $x+y$ (additive) and $x + y - xy$ (multiplicative) respectively. By Theorem \fref{KTheoryCobordism} (after setting $\beta = 1$) the cohomology rings of $\Z_m \otimes_\Laz \Omega^*$ are exactly $K^0$. Hence, if we twist $\ch_\Q^*$ to have the multiplicative formal group law, then the classifying map from $\Omega^*$ factors through $\Z_m \otimes_\Laz \Omega$ giving us the desired maps $K^0(X) \to \ch^*_\Q(X)$.

On the other hand, $\Q_m \otimes_\Laz \Omega^*$ may be twisted to have the additive formal group law, yielding a canonical map $\ch^*_\Q \to \Q_m \otimes_\Laz \Omega^*$. Similar reasoning as in the previous proof shows that this is the inverse to the canonical map $\Q_m \otimes_\Laz \Omega^* \to \ch_\Q^*$, which shows that the map $K^0_\Q (X) \to \ch^*_\Q(X)$ is invertible.
\end{proof}

\begin{rem}
Why is the above theorem a generalization of the usual Riemann-Roch? We identified twisted $\ch^*_\Q$ with $\Q_m \otimes_\Laz \Omega^*_\Q$.  The twist is induced by the sequence $\tau$ determined by the formula
\begin{equation*}
\sum_i b_i x^i = {x \over 1- e^{-x}}.
\end{equation*}
Indeed, the inverse Todd class of the zero section $s: X \to \Li$ is by definition
\begin{equation*}
\mathrm{Td}_{\tau}^{-1} (s) = \mathrm{Td}_{\tau} (\Li) =  {1 - e^{-c_1(\Li)} \over c_1(\Li)},
\end{equation*}
giving us the formula $c_1^{\tau}(\Li) = 1 - e^{-c_1(\Li)}$ for the twisted first Chern classes. Hence, in the $\tau$ twisted Chow theory, we have
\begin{align*}
c_1^{\tau}(\Li \otimes \Li') &= 1 - e^{- c_1(\Li) - c_1(\Li')} \\
&= (1 - e^{-c_1(\Li)}) + (1 - e^{-c_1(\Li')}) - (1 - e^{-c_1(\Li)}) (1 - e^{-c_1(\Li')} ) \\
&= c_1^{\tau}(\Li) + c_1^{\tau}(\Li') - c_1^{\tau}(\Li)c_1^{\tau}(\Li'),
\end{align*}
showing that $\ch^*_{\tau} \cong \Q_m \otimes_\Laz \Omega^*_\Q$.

The induced map $ch: K^0 \to \ch_\Q^*$ of cohomology theories does not commute with pushforwards along proper quasi-smooth maps $f: X \to Y$. This failure is measured by $\mathrm{Td}_\tau^{-1} (f)$ in the sense that
\begin{align*}
ch(f_!(\alpha)) &= ch(f_*(\alpha \bullet \theta(f))) \\
&= f_*(ch(\alpha) \bullet ch(\theta(f))) \\
&= f_*(ch(\alpha) \bullet \mathrm{Td}^{-1}_\tau(f) \bullet \theta'(f))\\
&= f_!(ch(\alpha) \bullet \mathrm{Td}^{-1}_\tau(f)),
\end{align*}
where $\theta$ is the orientation of $K$-theory and $\theta'$ the orientation of Chow theory.

Moreover, $ch$ is the usual Chern character. To see this, consider a line bundle $\Li$ on $X$. We know that $[\Li] = 1 - c_1(\Li^\vee)$ and that $ch(c_1(\Li^\vee)) = 1 - e^{c_1(\Li^\vee)}$ showing that $ch([\Li]) = e^{-c_1(\Li)}$. As every vector bundle can be pulled back from a smooth variety, and as splitting principle holds on smooth varieties, we obtain the usual formula for $ch(E)$.
\end{rem}

\begin{rem}
There should be a bivariant version of the above theorem where bivariant algebraic $K$-theory and bivariant Chow theory are rationally equivalent. We cannot prove this yet, as we don't know if bivariant cobordism specializes to bivariant $K$-theory.
\end{rem}

\subsubsection*{Comparison with Fulton's Chow ring of singular varieties}

In \cite{Ful2} Section 3, Fulton defines the Chow ring of arbitrary quasi-projective schemes as the left Kan extension of Chow ring restricted to smooth schemes. In other words, he sets 
$$\lch^*(X) := \colim_{X \to Y} \ch^*(Y)$$
where the colimit is taken over the opposite category of all smooth quasi-projective schemes under $X$ (in other words, the structure morphisms in the diagram are given by Gysin pullbacks). The above definition also makes sense when $X$ is a quasi-projective derived scheme if we take the colimit over the homotopy category of the opposite $\infty$-category of smooth quasi-projective $Y$ under $X$.

As the Chow rings $\ch^*$ of the previous subsection have arbitrary pullbacks, and as they coincide with the usual Chow rings when restricted to smooth quasi-projective schemes, we obtain a natural transformation (extending identity and commuting with pullbacks)
$$\eta_L: \lch^*(X) \to \ch^*(X)$$
for all quasi-projective derived schemes $X$. We also have the following

\begin{prop}\label{CHRationallyStableUnderLKE}
The natural transformation
$$\eta_L: \lch^*(X) \to \ch^*(X)$$
becomes an isomorphism after tensoring with $\Q$ for all quasi-projective classical schemes.
\end{prop}
\begin{proof}
Both groups are a target of a Chern character isomorphism
$$ch': K^0_\Q(X) \to \lch_\Q^*(X)$$
and 
$$ch: K^0_\Q(X) \to \ch^*_\Q(X),$$
and as all vector bundles can be pulled back from smooth schemes, the maps $ch$ and $\eta_L^\Q \circ ch'$ coincide. By two-out-of-three property of isomorphisms, also $\eta_L^\Q$ must be an isomorphism.
\end{proof}

In fact it should be very easy to prove that the Chern character $ch': K^0_\Q(X) \to \lch^*_\Q(X)$ should be an isomorphism for derived quasi-projective $X$ as well, extending the above result to derived schemes as well (the proof consists merely of redoing some arguments of \cite{Ful2} in derived algebraic geometry). However, the following question seems both interesting and nontrivial.

\begin{que}\label{IsCHStableUnderLKE}
Is the natural transformation
$$\eta_L: \lch^*(X) \to \ch^*(X)$$
always an isomorphism without passing to rational coefficients?
\end{que}

\Addresses

\end{document}